\documentclass[12pt,reqno]{amsart}

\usepackage{geometry}
\usepackage{amssymb}
\usepackage{amsmath}
\usepackage{latexsym}
\usepackage{txfonts}
\usepackage{mathtools}
\usepackage{xcolor}
\usepackage{color}

\usepackage[pagewise]{lineno}

\makeatletter
\@namedef{subjclassname@2020}{
  \textup{2020} Mathematics Subject Classification}
\makeatother

\usepackage[T1]{fontenc}

\newtheorem{theorem}{Theorem}[section]
\newtheorem{corollary}[theorem]{Corollary}
\newtheorem{lemma}[theorem]{Lemma}
\newtheorem{proposition}[theorem]{Proposition}

\theoremstyle{definition}
\newtheorem{definition}[theorem]{Definition}
\newtheorem{remark}[theorem]{Remark}

\numberwithin{equation}{section}

\newcommand\E{\mathbb{E}}
\newcommand\Z{\mathbb{Z}}

\newcommand\R{\mathbb{R}}

\newcommand\C{\mathbb{C}}

\newcommand\N{\mathbb{N}}

\newcommand\X{\mathcal{X}}
\newcommand\Y{\mathcal{Y}}

\newcommand\BD{\underline{\mathtt{BD}}}

\newcommand\st{\mathtt{st}}
\newcommand{\ep}{\varepsilon}

\newcommand\Aut{\operatorname{Aut}}

\newcommand\Baire{\mathcal{B}a}
\newcommand\Borel{\mathcal{B}o}

\newcommand\ConcProb{\mathbf{CncPrb}}
\newcommand\CHProb{{\mathbf{CHPrb}}}
\newcommand\ProbAlg{\mathbf{PrbAlg}}
\newcommand\OpProbAlg{{\mathbf{PrbAlg}}}
\newcommand\OpProbAlgG{{\mathbf{PrbAlg}_\Gamma}}

\newcommand\Stone{\mathtt{Conc}}
\newcommand\Inv{\mathtt{Inv}_\Gamma}

\begin{document}


\baselineskip=17pt


\title[An uncountable ergodic Roth theorem and applications]{An uncountable ergodic Roth theorem and applications}

\author[P. Durcik]{Polona Durcik}
\address{Polona Durcik\\ Schmid College of Science and Technology\\ Chapman University  \\
Orange \\ CA,  USA}
\email{durcik@chapman.edu}

\author[R. Greenfeld]{Rachel Greenfeld}
\address{Rachel Greenfeld\\ Department of Mathematics\\ University of California \\ 
Los Angeles \\
CA, USA}
\email{greenfeld@math.ucla.edu}

\author[A. Iseli]{Annina Iseli}
\address{Annina Iseli\\ Department of Mathematics\\ University of California \\ 
Los Angeles \\ 
CA, USA \newline \indent Department of Mathematics\\ University of Fribourg \\
Fribourg \\
Switzerland}
\email{annina.iseli@unifr.ch}

\author[A. Jamneshan]{Asgar Jamneshan}
\address{Asgar Jamneshan\\ Department of Mathematics\\ Ko\c{c} University \\ Istanbul \\
Turkey}
\email{ajamneshan@ku.edu.tr}

\author[J. Madrid]{Jos\'e Madrid}
\address{Jos\'e Madrid\\  Department of Mathematics\\ University of California \\ 
Los Angeles \\
CA, USA}
\email{jmadrid@math.ucla.edu}

\date{\today}

\begin{abstract}
We establish an uncountable  amenable ergodic Roth theorem, in which the acting group is not assumed to be countable and the space need not be separable. This generalizes a previous result of Bergelson, McCutcheon and Zhang, and complements a result of Zorin-Kranich.
We establish the following two additional results:
First, a combinatorial application about triangular patterns in certain subsets of the Cartesian square of arbitrary amenable groups, extending a result of Bergelson, McCutcheon and Zhang for countable amenable groups.
Second, a uniformity aspect in the double recurrence theorem for $\Gamma$-systems for arbitrary uniformly amenable groups $\Gamma$. 
Our uncountable Roth theorem is crucial in the proof of both of these results.
\end{abstract}



\maketitle

\section{Introduction}

A famous and deep theorem of Szemer\'edi \cite{szemeredi1975sets} asserts that any subset of the integers of positive upper density contains arbitrarily long arithmetic progressions.
The special case of this theorem for three term progressions was established earlier by Roth \cite{roth1952}.
In the pioneering article \cite{furstenberg1977ergodic}, Furstenberg related Szemer\'edi's theorem to a multiple recurrence theorem in ergodic theory which initiated a very fruitful development of applying dynamical methods to arithmetic combinatorics.
The focus of the current paper is on extensions, uniformity aspects, and combinatorial applications of the double recurrence theorem of Furstenberg which corresponds to the theorem of Roth on the combinatorial side.

Let us state Furstenberg's double recurrence theorem also called the ergodic Roth theorem in \cite{furstenberg1977ergodic}.
For a measurable space $(X,\mathcal{X})$, we let $\Aut(X,\mathcal{X})$ denote the group of bimeasurable point maps $f:X\to X$ where  the group law is given by composition of functions.
If $\mu$ is a probability measure on $(X,\mathcal{X})$, we denote by $\Aut(X,\mathcal{X},\mu)$ the subgroup of $\Aut(X,\mathcal{X})$ preserving $\mu$, that is all $f\in \Aut(X,\mathcal{X})$ such that $\mu(f^{-1}(E))=\mu(E)$ for all $E\in\mathcal{X}$. A measure-preserving action of the integers $\mathbb{Z}$ on $(X,\mathcal{X},\mu)$ is a group homomorphism $T:\mathbb{Z}\to\Aut(X,\mathcal{X},\mu)$ written as $n\mapsto T^n$.
Furstenberg's  double recurrence theorem \cite[Theorem 3.5]{furstenberg1977ergodic} states that for any such action $T$ and any $E\in\mathcal{X}$ of positive $\mu$-measure, the limit
\begin{equation*}
\lim_{N\to \infty} \frac{1}{N} \sum_{n=1}^N \mu(E\cap T^n E\cap T^{2n} E)
\end{equation*}
exists and is positive.
Using a correspondence principle \cite[Lemma 3.17]{furstenberg2014recurrence}, Furstenberg showed that this double recurrence theorem is equivalent to Roth's theorem.
In fact, Furstenberg's  result yields that the set of double recurrence times
\begin{equation*}
\{n\in\mathbb{Z}: \mu(E\cap T^n E\cap T^{2n} E)>0\}
\end{equation*}
is syndetic, that is intuitively speaking, has bounded gaps (a formal definition is given further below). In particular, this implies that there are infinitely many three term arithmetic progressions in a subset of positive upper density of the integers.
  Bergelson, Host and Kra \cite{bergelson-nilseq} significantly strengthen Furstenberg's double recurrence theorem by showing that
\begin{equation*}
  \{n\in\mathbb{Z}: \mu(E\cap T^n(E)\cap T^{2n}(E))>\mu(E)^3-\delta\}
\end{equation*}
 is syndetic for all $\delta>0$ under the additional hypothesis that the measure-preserving dynamical $\mathbb{Z}$-system is ergodic (their strengthening aligns with Khintchine's \cite{khintchine} strengthening of Poincare's recurrence theorem \cite{poincare}).
An important aspect of the Bergelson-Host-Kra result is that the lower bound depends only on the measure of $E$, but is otherwise uniform over all measure-preserving ergodic $\mathbb{Z}$-systems. In other terms, for every $\ep>0$, any measure-preserving ergodic $\mathbb{Z}$-system $(X,\mathcal{X},\mu,T)$ and all $E\in X$ with $\mu(E)\geq \ep$, we have that \begin{equation*} \{n\in\mathbb{Z}: \mu(E\cap T^n(E)\cap T^{2n}(E))>\ep^3-\delta\} \end{equation*}  is syndetic for all $\delta>0$. 

From this ergodic-theoretic perspective,
it is natural to ask how one can generalize Furstenberg's double recurrence theorem, its uniformity aspects, and its combinatorial consequences to the setting of other group actions. Furstenberg and Katznelson \cite{furstenberg1978ergodic}  establishes a double recurrence theorem for  $\mathbb{Z}^2$-actions. Moreover, Conze and Lesigne \cite{conles84} establish the $L^2$-convergence of the respective ergodic averages. On the other hand,
examples due to Bergelson and Hindman \cite{bergelson1992topological} and Bergelson and Leibman \cite{bergelson2004failure} show that a na\"ive translation of Furstenberg's and Furstenberg-Katznelson's double recurrence theorems to arbitrary countable amenable groups may fail.
Instead, Bergelson, McCutcheon and Zhang \cite{bergelson1997roth} propose to study the averages
\begin{equation*}
\frac{1}{|\Phi_n|}\sum_{\gamma\in\Phi_n} \mu(E\cap T^\gamma E\cap S^\gamma T^\gamma E)
\end{equation*}
where $S,T\colon \Gamma\to \Aut(X,\mathcal{X},\mu)$ are two commuting actions of an amenable group~$\Gamma$ and $(\Phi_n)$ is a F\o lner sequence for $\Gamma$. In fact, in the same article Bergelson, McCutcheon and Zhang establish existence and positivity of the limit of the above averages as $n\to \infty$. As a consequence, using a  version of Furstenberg's correspondence principle for countable amenable groups, they establish the following combinatorial application. Let $E$ be a subset of $\Gamma\times \Gamma$ with positive upper density with respect to some F\o lner sequence in $\Gamma\times \Gamma$. Then the set
\begin{equation*}
\{\gamma\in\Gamma: \text{ there exists } (a,b)\in\Gamma\times \Gamma \text{ with } ((a,b), (\gamma a, b), (\gamma a, \gamma b))\in E\}
\end{equation*}
is (left and right) syndetic. For a definition of F\o lner sequences (and nets) and the notions of syndeticity in general groups the interested reader is referred to Section \ref{sec:results} below.

Austin \cite{austin2016nonconventional} generalized both the convergence and double recurrence results of Bergelson, McCutcheon and Zhang to finitely many commuting actions of countable amenable groups  using his method of sated extensions.
Furthermore,
Chu and Zorin-Kranich \cite{lowerbound} establish a corresponding Khintchine-type theorem for two commuting ergodic actions of a countable amenable group, using techniques of Austin's sated extensions and relying on previous work by Chu \cite{chu-two} in the setting of two commuting ergodic measure-preserving $\Z$-actions. Finally, we also point out the recent paper by Moragues \cite{moragues}, where some versions of the Bergelson-Host-Kra uniformity results are obtained for finitely many commuting measure-preserving ergodic $\mathbb{Z}$-actions.
We stress that all the Khintchine-type uniformity results \cite{bergelson-nilseq,chu-two,lowerbound,moragues} work only in the class of ergodic systems.

Adapting arguments of Walsh \cite{walsh2012norm}, Zorin-Kranich \cite{zorin2016norm} extended the convergence results of Bergelson, McCutcheon and Zhang, and Austin to the actions of arbitrary, not necessarily countable, amenable groups acting on arbitrary, not necessarily separable spaces.
However, Zorin-Kranich's result does not provide much information about the limit object. In particular it does not entail multiple or even double recurrence.
In this context, it should also be mentioned that one easily obtains at least one double or multiple recurrence time for any group from Furstenberg's multiple recurrence theorem by restricting to a cyclic subgroup. However, this reduction does not yield syndeticity of multiple recurrence times which is relevant in combinatorial applications.

A main goal of this paper is to obtain syndeticity of double return
times for arbitrary amenable groups by extending the double recurrence theorem of Bergelson, McCutcheon and Zhang to uncountable amenable groups acting on arbitrary not necessarily separable spaces.
By an uncountable version of Furstenberg's correspondence principle, we also obtain an analogue of the aforementioned combinatorial application for uncountable amenable groups.
For this combinatorial application, we need to allow the underlying probability spaces to be inseparable since the shift systems required in Furstenberg's correspondence principle for uncountable groups are inseparable by construction.
Moreover, we derive a new uniformity aspect for the set of double return times in the amenable ergodic Roth theorem which also heavily relies on the uncountable/inseparable setting.

Some foundational aspects arising in the ergodic theory of uncountable groups and inseparable spaces was systematically investigated by the fourth author and Tao  in  \cite{jt-foundational,jt19,jt20} and by the fourth author in \cite{jamneshan2019fz}.  For example, in the area of multiple recurrence, the tool of disintegration of measures is used extensively. It is well known that in its classical form,  disintegration of measures fails for inseparable spaces in general. One of the major challenges is then to find a suitable alternative framework in which we find viable replacements for tools such as disintegration of measures which can help to meaningfully adapt the arguments from the countable setting. For further details, we refer the interested reader to \cite{jt-foundational,jt19,jt20,jamneshan2019fz}.

\subsection{Results}
\label{sec:results}
In order to state our results, we briefly introduce the framework of abstract measure preserving dynamical systems. See Section \ref{sec-canmodel} for a more detailed account.
The framework of uncountable groups yields the problem of unions of uncountably many null sets if we were to work with classical probability spaces. Therefore, we work with probability algebras instead.
A {\em probability algebra} is a tuple $(X,\mu)$ of an abstract $\sigma$-complete Boolean algebra $X=(X,\wedge,\vee,\bar\cdot,0,1)$ equipped with a probability measure $\mu$.\footnote{Probability algebras are also called measure algebras in the ergodic theory literature, e.g., see \cite{furstenberg2014recurrence,glasner2015ergodic}.}
Probability algebras can be thought of as point-free probability spaces since there is no underlying set {\em a priori}, i.e. not every abstract $\sigma$-complete Boolean algebra is a $\sigma$-algebra of subsets of a set, see \cite{loomis1947}.  Usually, there is also no loss of generality when working with probability algebras instead of classical probability spaces. Namely,  one can always quotient out the null ideal of a concrete probability space to obtain a probability algebra.

Similarly to concrete measure-preserving actions, abstract measure-preserving actions can be introduced as group homomorphisms into the automorphism group of a probability algebra. The automorphism group $\Aut(X,\mu)$ of a probability algebra $(X,\mu)$ consists of all Boolean isomorphisms $f:X\to X$ which are measure-preserving, namely $\mu(f(E))=\mu(E)$ for every $E\in X$ and where the group law is given by composition of Boolean homomorphisms.
Now let $\Gamma$ be an arbitrary discrete, not necessarily countable group. We define a \emph{probability algebra $\Gamma$-dynamical system} to be a triple $(X,\mu,T)$ where $T:\Gamma\to\Aut(X,\mu)$ is a group homomorphism. We say that $T$ is an abstract action. Given a second abstract action $S$, we say that $T$ and $S$ commute if $T^\gamma \circ S^{\gamma'}(E)=S^{\gamma'}\circ T^\gamma(E)$ for all $E\in X$ and $\gamma,\gamma'\in \Gamma$, where the symbol $\circ$ denotes composition of maps.
We call the quadruple $(X,\mu,T,S)$ an \emph{abstract Roth $\Gamma$-dynamical system}.

Recall that a (left) F\o lner net for $\Gamma$ is  a net $(\Phi_\alpha)_{\alpha\in A}$ of non-empty finite subsets of $\Gamma$ such that
    \begin{equation*}
    \lim_{\alpha \in A} \frac{|\Phi_\alpha \Delta \gamma \Phi_\alpha|}{|\Phi_\alpha|}\to 0
    \end{equation*}
    for all $\gamma\in \Gamma$, where $\Delta$ denotes set symmetric difference.
A discrete group $\Gamma$ is said to be amenable if it has a F\o lner net\footnote{See Appendix \ref{sec-amsyn} for equivalent definitions of amenability.}.

Now we are ready to state our first main result, which is an uncountable version of the ergodic Roth theorem of Bergelson, McCutcheon and Zhang for discrete  amenable groups.
\begin{theorem}[Ergodic Roth theorem for arbitrary amenable groups]\label{thm-uncountableroth}
    Let $\Gamma$ be an arbitrary amenable discrete group. Let $(X,\mu,T,S)$ be an arbitrary abstract Roth $\Gamma$-dynamical system.
    Then for every $E\in X$ and left F\o lner net $(\Phi_\alpha)_{\alpha\in A}$ for $\Gamma$, the limit
    \begin{equation}
        \label{mainlimit}
        \lim_{\alpha\in A}\frac{1}{|\Phi_\alpha|}\sum_{\gamma\in \Phi_\alpha} \mu(E\wedge T^\gamma (E) \wedge S^\gamma T^\gamma(E))
    \end{equation}
    exists and is independent of the choice of the left F\o lner net.
    Moreover, the limit is positive whenever $\mu(E)>0$.
    \end{theorem}
From the uncountable ergodic Roth theorem we deduce several applications.  In these applications, the below corollary (Corollary~\ref{maincor}) is crucial. We recall the notion of syndeticity:
  a subset $E$ of a discrete group $\Gamma$ is said to be (left) syndetic if finitely many (left) shifts of $E$ cover all of $\Gamma$, more explicitly, if  there exists a finite set $F$ in $\Gamma$ such that $\bigcup_{\gamma \in F} \gamma E=\Gamma$.

    \begin{corollary}\label{maincor}

Suppose that $E\in X$ with $\mu(E)>0$ in the setting of Theorem \ref{thm-uncountableroth}.
Then there exists $\delta>0$ such that
\begin{equation*}
\{\gamma\in \Gamma : \mu(E\wedge T^\gamma E\wedge S^\gamma T^\gamma E)>\delta\}
\end{equation*}
is (left) syndetic.
\end{corollary}

Analogous statements hold if in Theorem \ref{thm-uncountableroth} and Corollary \ref{maincor} we   replace the notions of left F\o lner nets and  left syndeticity by the analogous notions of right F\o lner nets and  right syndeticity, respectively.

\subsection*{Applications}
First, we recall the notion of an  {\it invariant mean}.  An  {invariant mean} for $\Gamma$ is a positive linear functional $m: \ell^\infty(\Gamma)\to \R$ with the properties that $m(1)=1$ and $m(\gamma f)=m(f)$, where $(\gamma f)(\gamma'):=f (\gamma^{-1}\gamma')$ and $\gamma\in \Gamma,f\in \ell^{\infty}(\Gamma)$.
It is well know that a  discrete group is amenable if and only if it admits an invariant mean, e.g. see \cite{paterson}.

Using an uncountable version of  the Furstenberg correspondence principle we  deduce the following combinatorial result on   triangular patterns in $\Gamma\times\Gamma$, where $1_\Lambda$ denotes the characteristic function of a set $\Lambda$.
\begin{theorem}\label{prop-triangular}
Let $\Gamma$ be a discrete amenable group and  $m:\ell^\infty(\Gamma\times\Gamma)\to \R$ be an invariant mean. Suppose that $\Lambda\subset \Gamma \times \Gamma$ satisfies $m(1_\Lambda)>0$. Then the set
\begin{equation*}
\{\gamma\in \Gamma :  \text{ there exists }  (\theta,\zeta)\in \Gamma\times \Gamma \text{ such that } (\theta,\zeta),(\gamma \theta,\zeta),(\gamma \theta,\gamma \zeta)\in \Lambda \}
\end{equation*}
is syndetic.
 \end{theorem}

Our next application is a uniformity result. For the sake of better understanding, we first state it in the case of a single $\mathbb{Z}$-action (Theorem~\ref{thm-Zsyndec}) before we formulate it in the most general form (Theorem~\ref{thm-syndec}).
The $\mathbb{Z}$-case is well known, e.g.~see \cite{bergelson-uniform}. 
The focus is on a quantitative uniformity aspect of the syndeticity of the set of double return times
\begin{equation*} \{n\in \Z\colon \mu(E \cap T^n(E)\cap T^{2n}(E))>0\}\end{equation*}
for a set $E\in \mathcal{X}$, where $(X,\mathcal{X},\mu,T)$ is an arbitrary measure-preserving $\Z$-dynamical system.
A suitable way to quantify syndeticity uses the notion of lower Banach density.
Recall that the lower Banach density of a subset $A \subset \Z$ is defined as
\begin{equation}
    \label{def:lowerbd-Z}
    \underline{\mathtt{BD}}_\Z(A)=\liminf_{b-a \to \infty}\frac{|A\cap  \{a,a+1,\ldots,b\}|}{b-a+1}.
\end{equation}
We obtain the following result.
\begin{theorem}\label{thm-Zsyndec}
For every $\ep>0$ there exist $\delta,\eta>0$ (depending only on $\ep$) such that for any measure-preserving $\Z$-dynamical system $(X,\mathcal{X},\mu,T)$ and every $E\in \mathcal{X}$ with $\mu(E)\geq\ep$,
\begin{equation*}
\underline{\mathtt{BD}}_\Z(\{n\in \Z\colon \mu(E \cap T^n(E)\cap T^{2n}(E))>\delta\})>\eta.
\end{equation*}
\end{theorem}

\begin{remark}
The crucial of this uniformity result is that we can choose $\delta,\eta$ uniformly over \emph{all} measure-preserving $\Z$-dynamical systems and \emph{all} measurable sets $E\in \X$ depending \emph{only} on the size of the measure of $E$.
\end{remark}

\begin{remark}
In contrast to the previously discussed uniformity results, we do not need to assume ergodicity. In fact, it is shown in \cite[Theorem 2.1]{bergelson-nilseq} that a Khintchine-type uniform lower bound fails if one removes the hypothesis of ergodicity.
\end{remark}

We generalize Theorem \ref{thm-Zsyndec} from $\Z$-measure-preserving dynamical systems to  all Roth-type $\Gamma$-measure-preserving dynamical systems where $\Gamma$ is {uniformly amenable}, or more generally belongs to a uniformly amenable set of groups. The notion of uniform amenability was introduced  by Keller \cite{keller}.
It can be viewed as a uniform version of the F\o lner  condition,  in the following sense. A discrete group $\Gamma$ is said to be \emph{uniformly amenable} if there exists a function $F:\mathbb{N}\times (0,1)\to \mathbb{N}$ such that for every set $\Psi\subset \Gamma$ with $|\Psi|\leq n$ and $0<\ep<1$ there exists a set $\Phi\subset\Gamma$ with $|\Phi|\leq F(n,\ep)$ such that
\begin{equation}
    \label{def:unifam}
    \max_{\gamma\in \Psi} |\Phi \Delta \gamma\Phi|\leq \ep |\Phi|.
\end{equation}
More generally, a set $\mathcal{G}$ of discrete groups is said to be uniformly amenable if there exists a function  $F:\mathbb{N}\times (0,1)\to \mathbb{N}$ such that each group $\Gamma\in \mathcal{G}$ is uniformly amenable with respect to $F$.  A more detailed account on  uniform amenability can be found in Appendix \ref{sec-amsyn}.

In order to state this more general uniform syndeticity result, we also introduce the notion of lower Banach density for a subset $\Lambda$ of a discrete amenable group $\Gamma$ by setting
\begin{equation}
    \label{def:lbd-gen}
    \underline{\mathtt{BD}}_\Gamma(\Lambda)\coloneqq \inf\{ \nu(\Lambda): \nu \text{ is an invariant finitely additive probability measure} \},
\end{equation}
where a   finitely additive probability measure $\nu:\mathcal{P}(\Gamma)\to [0,1]$ is called \emph{invariant}  if $\nu(\gamma \Lambda)=\nu(\Lambda)$ for all $\gamma\in\Gamma$ and $\Lambda\subset \Gamma$, where $\gamma \Lambda=\{\gamma \tilde{\gamma}:\tilde{\gamma}\in \Lambda\}$\footnote{That the definition given in \eqref{def:lbd-gen} is equivalent to the definition above \eqref{def:lowerbd-Z} for $\Gamma=\mathbb{Z}$ is discussed in Appendix \ref{appendix_B1}.}.

Using these notions, we establish an extended version of  Theorem \ref{thm-Zsyndec}. 
We show the existence of a lower bound on the degree of syndeticity uniformly over a class of Roth-type measure-preserving dynamical systems for a uniformly amenable set of groups, as follows.
\begin{theorem}\label{thm-syndec}
Let $\mathcal{G}=\mathcal{G}(F)$ be a uniformly amenable set of groups and let $\ep>0$. Then there exist $\delta,\eta>0$,  depending only on $\ep$ and $\mathcal{G}$, such that for every $\Gamma\in\mathcal{G}$, any abstract Roth $\Gamma$-dynamical system $(X,\mu,T,S)$, and each $E\in X$ with $\mu(E)\geq \ep$,
\begin{equation*}
\BD_\Gamma(\{\gamma\in \Gamma\colon \mu(E \wedge T^\gamma(E)\wedge S^\gamma T^{\gamma}(E))>\delta\})>\eta.
\end{equation*}
\end{theorem}

\subsection{Proof methods}
 The proof of the multiple recurrence statement in Theorem \ref{thm-uncountableroth} follows the general outline of  \cite[Theorem 5.2]{bergelson1997roth}.
    However, several adaptations of the strategy in \cite{bergelson1997roth} are required in our uncountable, inseparable, and point-free framework. In particular, in this framework we are lacking a classical disintegration of measures and related tools such as direct integrals of Hilbert bundles. To fill in these gaps, we follow the approach recently developed in \cite{jamneshan2019fz,jt19,jt20,jt-foundational} by the fourth author and Tao.
    More precisely, we work with the canonical model of an abstract Roth $\Gamma$-system which is a compact Hausdorff space (in fact, a Stonean space) equipped with a Baire-Radon probability measure which is invariant under the action of $\Gamma$ by homeomorphsims.
    We review the construction of the canonical model and mention related references in Section \ref{sec-canmodel}.
    One immediate useful consequence of the canonical model is that it leads to a canonical disintegration for abstract factor maps.
    We can use the canonical disintegration to define {\it relatively independent products}.
    The relatively independent product is a relevant construction in order to identify the characteristic factors for the abstract Roth $\Gamma$-systems.
    These characteristic factors are the largest compact factor of the abstract $\Gamma$-systems $(X,\mu,T)$ and $(X,\mu,ST)$ over the invariant factor of $(X,\mu,S)$. Since we do not assume that $(X,\mu,S)$ is ergodic, these compact factors are not necessarily the Kronecker factors.
    We call them the conditional Kronecker factors.
    In Lemma \ref{lem-characteristic}, we establish that the conditional Kronecker factors are characteristic for the abstract Roth-type non-conventional averages. Hence we can project onto these factors and it suffices to establish the multiple recurrence statement in Theorem \ref{thm-uncountableroth} for these projections.
    A step in proving the latter multiple recurrence is a finite dimensional approximation of the $\Gamma$-orbits of functions in certain finitely generated $\Gamma$-invariant $L^\infty$ submodules with respect to the actions of $T$ and $ST$ respectively.
    In the countable-separable framework of \cite[Section 5]{bergelson1997roth}, this approximation is achieved by using direct integrals of Hilbert bundles and measurable selection techniques, e.g., see \cite[Chapter 9]{glasner2015ergodic} for a textbook reference.
   These tools are not available in our setting  even after passing to concrete models. The reason for this is that all these models are typically highly inseparable and the acting group is still uncountable. Therefore, we rely on conditional analysis techniques as developed in \cite{cheridito2015conditional,filipovic2009separation,drapeau2016algebra} instead.
    Our finite dimensional approximation is based  on a conditional Gram-Schmidt process and a conditional Heine-Borel covering lemma which we prove in Appendix \ref{sec-canalysis}. Their statements are inspired by earlier results in the conditional analysis literature.

To deduce Theorem \ref{prop-triangular}, we first establish a version of the Furstenberg correspondence principle for uncountable amenable discrete groups. Then we follow the outline of
 \cite[Theorem 6.2]{bergelson1997roth}.

    The proof of Theorem \ref{thm-Zsyndec} is subsumed into the proof of Theorem \ref{thm-syndec}.
  The latter relies on ultralimit analysis, which is discussed in Appendix \ref{sec-ultra}.
To prove Theorem \ref{thm-syndec}, we proceed by contradiction by assuming that there is a sequence $(X_n,\mu_n,T_n,S_n)$ of abstract Roth $\Gamma_n$-dynamical systems with $\Gamma_n\in\mathcal{G}$ and $E_n\in X_n$ with  $\mu_n(E_n)\geq \ep$ such that
\begin{equation*}
\BD_{\Gamma_n}(\{\gamma\in \Gamma_n\colon \mu(E_n \wedge T^\gamma(E_n)\wedge S^\gamma T^{\gamma}(E_n))>1/n\})\leq 1/n.
\end{equation*}
  We can then form the ultralimit system $(X_*,\mu_*,\Gamma_*,T_*,S_*)$ from this sequence of systems. The uniform amenability hypothesis is relevant here to verify that the ultralimit group $\Gamma_*$ is amenable.
The difficult and crucial step is to relate the sequence of lower Banach densities $\BD_{\Gamma_n}$, $n\in \N$, with the lower Banach density~$\BD_{\Gamma_*}$ of the ultraproduct group. We establish a useful relation by employing a Loeb measure construction and applying a Hahn-Banach extension theorem for invariant means due to Silverman \cite{silverman1,silverman2}.  Then we apply Corollary \ref{maincor} which yields syndeticity of multiple return times for the ultralimit system $(X_*,\mu_*,\Gamma_*,T_*,S_*)$. This leads to the desired contradiction. Notice that in the ultralimit system $(X_*,\mu_*,\Gamma_*,T_*,S_*)$, the probability algebra $(X_*,\mu_*)$ is almost never separable and the group $\Gamma_*$ is almost never countable even if $(X_n,\mu_n)$ were separable probability algebras and the $\Gamma_n$ were countable groups. Thus Theorem \ref{thm-uncountableroth} is essential to establish uniform syndeticity even for the class of systems where a countable group acts on a separable probability algebra.

    \subsection{Organization of the paper}
    In Section \ref{sec-canmodel}, we introduce the formal setup for this paper. In particular, this covers probability algebras, measure preserving dynamical systems, and canonical models and disintegration. In Section \ref{sec-convergence}, we discuss characteristic factors and non-conventional averages and then proceed to proving Theorem \ref{thm-uncountableroth}, Corollary \ref{maincor}, as well as Proposition \ref{prop-triangular}. Section \ref{sec-uniformity} is about the proof of Theorem \ref{thm-syndec}.
    In Appendix \ref{sec-BA}, we review Boolean algebras and the Stone representation theorem.
    In Appendix \ref{sec-amsyn}, we review amenability, uniform amenability, and relate amenability to syndeticity.
    In Appendix  \ref{sec-ultra}, we record an ultralimit construction for a sequence of abstract Roth dynamical systems. Finally, in Appendix \ref{sec-canalysis} we prove a conditional Heine-Borel covering lemma used in the proof of Theorem \ref{thm-uncountableroth}.

  \subsection{Notation and conventions}

  Let $X$ be a set.
  We denote by $\mathcal{P}(X)$ its power set.
  If $X$ is finite, we denote by $|X|$ its cardinality.
   All F\o lner nets are understood to be left F\o lner nets. Similarly, syndeticity is understood as left syndeticity. By symmetry,  all related results in this paper remain true if we replace left F\o lner nets with right F\o lner nets and left syndeticity with right syndeticity.
   Suppose $\Gamma$ is a group, $X$ is a set, and $T:\Gamma\times X\to X$ is a group action. Then we stipulate the convention $T^\gamma(f)=f\circ T^{\gamma^{-1}}$ for a function $f:X\to \C$. This implies that $T^{\gamma_1\gamma_2}(f)=f\circ T^{\gamma_2^{-1}\gamma_1^{-1}}=T^{\gamma_2}(T^{\gamma_1}(f))$ for all $\gamma_1,\gamma_2\in \Gamma$. Due to this property, the induced action on some function space is called {\it antihomomorphic}.
      With the above convention, we also have $T^\gamma(1_E)=1_{T^\gamma(E)}$ for a subset  $E\subset X$.
      For a compact Hausdorff space $X$, we denote by $\Baire(X)$ the Baire $\sigma$-algebra of $X$, i.e. the smallest $\sigma$-algebra generated by the real-valued continuous functions.

      Throughout the paper, all equalities and inequalities between measurable functions and measurable sets are understood in an almost sure sense. As we will deal with different measure spaces at the same time (usually factors or extensions of each others), the almost sure statement is understood with respect to the obvious probability measure. For example, $f=g$ for two functions on a common probability space $(X,\X,\mu)$ is understood as $\mu(\{x:f(x)=g(x)\})=1$.  Similarly for $f<g$, $f\leq g$, $E=F$, and $E\subset F$ where $E,F\in \X$. If a certain property is said to hold \emph{on} a measurable set $E$, then we mean that $\mu(E\Delta F)=0$ for the measurable set $F$ on which this property is satisfied, where $\Delta$ denotes symmetric set difference.

   \section*{Acknowledgments}
RG was partially supported by the Eric and Wendy Schmidt Postdoctoral Award.
AI was partially supported by the Swiss National Science Foundation (project no.\ 181898). 
AJ was supported by DFG-research fellowship JA 2512/3-1. 
The authors would like to express their gratitude to Terence Tao for inspiring this work, for helpful discussions, and for his encouragement. 
The authors thank an anonymous referee for useful comments and suggestions.

   \section{Canonical models and canonical disintegrations}\label{sec-canmodel}

In this section we discuss the formal setup in which our results are phrased and proven. Experts may be able to skip portions of this section. We start by introducing canonical models in order to develop some basics of abstract measure theory and define canonical disintegrations. This will allow us to define relatively independent products of probability algebras.
A more extensive treatment of these topics can be found in \cite{jt-foundational}.

Glasner defines in \cite[Definition 2.14]{glasner2015ergodic} a measure-preserving dynamical system to be a tuple $(X,\mu,\Gamma)$ where $(X,\mu)$ is a \emph{separable} probability algebra and $\Gamma$ is a \emph{countable} group of automorphisms of $(X,\mu)$.
In \cite[Theorem 2.15.1]{glasner2015ergodic}, Glasner then shows that any such measure-preserving dynamical system can be \emph{modeled} by a Cantor measure-preserving system $(\tilde{X},\Borel(\tilde{X}),\tilde{\mu},\tilde{\Gamma})$, where $\tilde{X}=\{0,1\}^\mathbb{N}$ is the Cantor space, $\Borel(\tilde{X})$ is its Borel $\sigma$-algebra, $\tilde{\mu}$ is a Borel probability measure constructed from $\mu$, and $\tilde{\Gamma}$ is a countable group of $\tilde{\mu}$-preserving homeomorphisms of $\tilde{X}$.
To be modeled means here that both systems are isomorphic in the category of probability algebra dynamical systems (see Definition \ref{def-spaces} for a definition of this category).
See also Furstenberg \cite[Section 5.2]{furstenberg2014recurrence} for a closely related construction.

We introduce next the definition of a topological model for measure-preserving dynamical systems $(X,\mu,\Gamma)$ where $(X,\mu)$ is a not necessarily separable probability algebra and $\Gamma$ is a not necessarily countable group.
We call this compact Hausdorff model \emph{canonical} since it satisfies suitable universality properties, we refer the interested reader to \cite[Proposition 7.6]{jt-foundational} for details.
Closely related models haven been suggested either implicitly or explicitly at several occasions in the literature, e.g., see \cite{fremlinvol3,EFHN} and \cite{jt-foundational} for a list of other references.

We define the three categories of dynamical systems employed in this work.

\begin{definition}[Concrete and abstract measure-preserving dynamical systems] \label{def-spaces}
Let $\Gamma$ be a discrete group.
\begin{itemize}
    \item[(i)] (The category $\ConcProb_\Gamma$ of concrete measure-preserving dynamical systems) A \emph{concrete probability space} is a triple $(X,\X,\mu)$ where $X$ is a set, $\X$ is a $\sigma$-algebra of subsets of $X$, and $\mu:\X\to[0,1]$ is a countably additive probability measure.  A \emph{concrete measure-preserving map} from a concrete probability space $(X,\X,\mu)$ to another $(Y,\Y,\nu)$ is a measurable function $f:X\to Y$ such that $f_\#\mu=\nu$ where
\begin{equation*}
f_\#\mu(E):= \mu(f^{-1}(E))=\nu(E)
\end{equation*}
    for all $E\in \Y$. We denote by $\ConcProb$ the category of concrete probability spaces.

    For a concrete probability space $(X,\X,\mu)$, we denote by $\Aut(X,\X,\mu)$ its automorphism group, i.e. the group of all bi-measurable functions $f:X \to X$ such that both $f$ and its inverse $f^{-1}$ are measure-preserving.
    A \emph{concrete measure-preserving dynamical system} is a tuple $(X,\X,\mu,T)$ where $T$  is a group homomorphism $\Gamma\to \Aut(X,\X,\mu)$, $\gamma\mapsto T^\gamma$.  We call $T$ a \emph{concrete action}. Given a second concrete measure-preserving dynamical system $(Y,\Y,\nu,S)$, a concrete measure-preserving function $\pi:X\to Y$ is called a \emph{concrete factor map} if $S^\gamma\circ \pi(x)=\pi\circ T^\gamma(x)$ for all $x\in X$ and every $\gamma\in \Gamma$. In this case we call $(X,\X,\mu,T)$ a \emph{concrete extension} of $(Y,\Y,\nu,S)$, and $(Y,\Y,\nu,S)$ a \emph{concrete factor} of $(X,\X,\mu,S)$.
    We denote by $\ConcProb_\Gamma$ the category of concrete measure-preserving dynamical systems.
    \item[(ii)] (The category $\CHProb_\Gamma$ of topological measure-preserving dynamical systems) A \emph{compact Hausdorff probability space} is a tuple $(X,\Baire(X),\mu)$, where $X$ is a compact Hausdorff space with Baire $\sigma$-algebra $\Baire(X)$ and $\mu:\X\to[0,1]$ is a Baire-Radon probability measure (a probability measure satisfying the regularity property $\mu(E)=\sup\{\mu(F): F\in \X, F\subset E, F\text{ compact }G_\delta\}$ for all $E\in \Baire(X)$).  A morphism in the category of compact Hausdorff probability spaces is a measure-preserving continuous function.
    We name this category $\CHProb$.
    Similarly to $\ConcProb_\Gamma$, we define the dynamical category $\CHProb_\Gamma$ where now the automorphism group $\Aut(X,\Baire(X),\mu)$ consists of measure-preserving homeomorphisms of compact Hausdorff probability spaces.
    \item[(iii)] (The category $\OpProbAlgG$ of probability algebra dynamical systems)
   Let's start by defining the category $\ProbAlg$ first.
   A \emph{probability algebra} is a tuple $(X,\mu)$ where $X$ is a $\sigma$-complete Boolean algebra (see Appendix \ref{sec-BA} for a basic introduction to Boolean algebras) and $\mu: X\to [0,1]$ is a countably additive probability measure, that is, $\mu(\bigvee_{i=1}^\infty E_i)=\sum_{i=1}^\infty \mu(E_i)$ for every countable family $(E_i)$ of pairwise disjoint elements in $X$, $\mu(1)=1$ and $\mu(E)=0$ if and only if $E=0$. We define a  \emph{probability algebra morphism} from a probability algebra $(X,\mu)$ to another $(Y,\nu)$ to be a Boolean homomorphism\footnote{We implicitly use the dual category here to keep the canonical model functor covariant.} $f:Y \to X$ such that $\mu(f(E))=\nu(E)$ for all $E\in Y$. Notice that we do not stipulate that $f$ is a Boolean $\sigma$-homomorphism since this follows automatically: If $(E_n) $ is a countable family of elements of $Y$ with union $E=\bigvee  E_n$, then $\nu(E\backslash \bigvee_{n= 1}^N E_n)=\mu(f(E)\backslash \bigvee_{n= 1}^N f( E_n))\to 0$ which implies $f(E)=\bigvee  f(E_n)$ because $\mu(E)=0$ if and only if $E=0$.
   Probability algebras and probability algebra morphisms form the category $\ProbAlg$ of probability algebras.

    The automorphism group in $\ProbAlg$ of a probability algebra is the group $\Aut(X,\mu)$ consisting of all measure-preserving (opposite) Boolean isomorphisms of $X$ to itself. A \emph{probability algebra dynamical system} is a tuple $(X,\mu,T)$ where $T:\Gamma \to \Aut(X,\mu)$ is a group homomorphism $\gamma\mapsto T^\gamma$.  We call $T$ an \emph{abstract action}. A morphism from a probability algebra dynamical system $(X,\mu,T)$ to another $(Y,\nu,S)$ is a $\ProbAlg$-morphism $\pi:X\to Y$ such that $T^\gamma\circ \pi(E)=\pi\circ S^\gamma(E)$ for all $E\in Y$ and  $\gamma\in \Gamma$.  We call $\pi$ also an \emph{abstract extension map}, $Y$ an \emph{abstract factor} of $X$, and $X$ an \emph{abstract extension} of $Y$.
    \end{itemize}
    \end{definition}

    We describe next the two important processes of how to canonically associate a probability algebra system to any concrete measure-preserving system, and conversely how to canonically associate to any probability algebra system a topological measure-preserving system.
    Throughout we fix a discrete group $\Gamma$.

Let $(X,\X,\mu,T)$ be a $\ConcProb_\Gamma$-system and let $\mathcal{N}_\mu=\{E\in \X: \mu(E)=0\}$ denote the ideal of null sets of $(X,\X,\mu)$. Then the quotient Boolean algebra $X_\mu:=\X/_{\mathcal{N}_\mu}$, resulting from identifying $E,F\in\mathcal{X}$ whenever $\mu(E\Delta F)=0$, is $\sigma$-complete and we have a canonical Boolean $\sigma$-epimorphism $\pi: \X\to X_\mu$ which associates to each $E\in \X$ its equivalence class $[E]$ in $X_\mu$.
We define the associated probability algebra measure $\bar\mu: X_\mu \to [0,1]$ by $\bar\mu([E]):=\mu(E)$.
For any $\gamma\in\Gamma$, we define $\bar{T}^\gamma: X_\mu\to X_\mu$ by $\bar{T}^\gamma([E]):= \pi((T^\gamma)^{-1}(E))$. We obtain a probability algebra action $\bar{T}:\Gamma\to \Aut(X_\mu,\bar{\mu})$ such that $(X_\mu,\bar{\mu},\bar{T})$ is a probability algebra dynamical system. Passing to the dual category provides us with a canonical choice of a $\OpProbAlgG$-system associated to $(X,\X,\mu,T)$. Of course, the same construction works for any $\CHProb_\Gamma$-system as well.
This combined abstraction and deletion process, i.e. when we delete the null sets and with it the point-set structure of the measurable space $(X,\X)$, is functorial. In particular, any $\ConcProb_\Gamma$-factor map is associated to a $\OpProbAlgG$-factor map.  However this functor is not injective on objects; for example the associated probability algebra cannot distinguish between a concrete probability space and its measure-theoretic completion.

The canonical model functor reverses this process by associating to any $\OpProbAlgG$-dynamical system a canonical $\CHProb_\Gamma$-dynamical system.
We sketch one of the two constructions of the canonical model functor given in \cite[Sections 7, 9]{jt-foundational} which is based on the Stone representation theorem (the latter theorem is recalled in Appendix \ref{sec-BA}).
Let $(X,\mu,T)$ be a $\OpProbAlgG$-dynamical system.  Let $\Stone(X)$ denote the Stone space of the Boolean algebra $X$ and equip it with the Baire $\sigma$-algebra $\Baire(\Stone(X))$.
We define the measure of $E\in \Baire(\Stone(X))$ to be the measure of the unique element of $X$ that generates a clopen subset of $\Stone(X)$ that differs from~$E$ by a Baire-meager set. It can be checked that this measure is a Baire-Radon probability measure and  we denote it by $\mu_{\Stone(X)}$.
For $\gamma\in \Gamma$, we define  ${T^\gamma_{\Stone(X)}:=\Stone(T^\gamma_X):\Stone(X)\to \Stone(X)}$ to be the unique homeomorphism obtained by applying the Stone functor to the opposite Boolean isomorphism ${T^\gamma:X \to X}$. In particular, the inverse image of Baire-meager sets under $T^\gamma_{\Stone(X)}$ are Baire meager and $T^\gamma_{\Stone(X)}$ preserves the Baire-Radon probability measure  $\mu_{\Stone(X)}$.  Hence, $(\Stone(X),\Baire(\Stone(X)),\mu_{\Stone(X)}, ,T_{\Stone(X)})$ is a $\CHProb_\Gamma$-system that we called the \emph{canonical model} of $(X,\mu, T)$.
This correspondence is again functorial, in particular $\OpProbAlgG$-factor maps are mapped to $\CHProb_\Gamma$-factor maps.
In fact, we have a stronger functorial property: If we compose the canonical model functor with the combined abstraction and deletion functor we obtain the identity functor on $\OpProbAlgG$ (up to natural isomorphisms).

We can use the canonical model functor to introduce $L^p$ spaces and integration on $\OpProbAlg$-spaces $(X,\mu)$ by just defining
\begin{equation*}
L^p(X):=L^p(\Stone(X))
\end{equation*}
for $1 \leq p\leq \infty$, and defining the integral of $f\in L^1(X)$ to be $\int_{\Stone(X)} f d\mu_{\Stone(X)}$. One can also define abstract $L^p$ spaces on probability algebras (or more generally, on measure algebras) directly without invoking a canonical model, see \cite{fremlinvol3}. One can then show that these abstract $L^p$ spaces are isomorphic (as Banach and Riesz spaces) to the ones defined above (see \cite[Remark 9.13]{jt-foundational} for comparison). Given a $\OpProbAlg$-dynamical system $(X,\mu,T)$,  define the Koopman operator $T^\gamma:L^p(X)\to L^p(X)$ by
\begin{equation*}
T^\gamma(f):=f\circ T^{\gamma^{-1}}_{\Stone(X)}, \quad f\in L^p(X)
\end{equation*}
for all $\gamma\in \Gamma$.

Given a $\ConcProb$-space $(X,\X,\mu)$, we also have the identifications
\begin{equation*}
L^p(X)\equiv L^p(X_\mu)= L^p(\Stone(X_\mu))
\end{equation*}
as Riesz and Banach spaces.
We will freely make use of these identifications in the sequel.

If $\pi:(X,\mu,T)\to (Y,\nu,S)$ is a $\OpProbAlgG$-factor map, then we have the pullback map $\pi^*:L^p(Y)\to L^p(X)$ defined by $\pi^*(f):= f\circ \Stone(\pi)$ which is easily seen to be an isometry. Thus we can identify $L^p(Y)$ with the closed invariant subspace $\pi^*(L^p(Y))$ in $L^p(X)$. In the case of $p=2$, the pullback map $\pi^*$ induces a conditional expectation operator $\E(\cdot\,|\, Y):L^2(X)\to L^2(Y)$ by defining
\begin{equation*}
\E(f \,|\, Y):=\E( f\, |\, \pi^*(L^2(Y)))
\end{equation*}
where $\E( f\, |\, \pi^*(L^2(Y)))$ is the orthogonal projection onto $\pi^*(L^2(\Stone(Y)))$ seen as a closed subspace of  $L^2(\Stone(X))$. Furthermore, we have $\E( T^\gamma(f)\, |\, Y)= S^\gamma(\E( f\,|\, Y))$ for all $f\in L^2(\Stone(X))$.

As a first example of a $\OpProbAlgG$-factor we have the $\OpProbAlgG$-invariant factor.
Let $(X,\mu,T)$ be a $\OpProbAlgG$-dynamical system.
Then we define the $\OpProbAlgG$-invariant factor to consist of the $\sigma$-complete Boolean algebra
\begin{equation*}
\Inv(X,\mu,T):=\{E\in X :  T^\gamma(E)= E \forall \gamma\in \Gamma\}
\end{equation*}
equipped with the probability measure $\mu$ and the restriction of the action $T$ to $\Inv(X,\mu,T)$ which is just the trivial action.
We call $(\Inv(X,\mu,T),\mu, T)$ the $\OpProbAlgG$-\emph{invariant factor} of $(X,\mu,T)$, where the factor map $\pi:X\to \Inv(X,\mu,T)$ is the canonical projection.
The invariant factor is a functor from the category $\OpProbAlgG$ to itself.  Namely,  if $\pi:(X,\mu,T)\to (Y,\nu,S)$ is a $\OpProbAlgG$-extension, then this yields an induced $\OpProbAlgG$-factor map
\begin{align*}
\pi:(\Inv(X,\mu,T),\mu,T)\to (\Inv(Y,\nu,S),\nu,S).
\end{align*}
We can combine the invariant factor functor with the canonical model functor to find a canonical representation of the invariant factor and the canonical projection in $\CHProb_\Gamma$.
A $\OpProbAlgG$-system $(X,\mu,T)$ is said to be \emph{ergodic} if $\Inv(X,\mu,T)$  is the trivial algebra $\{0,1\}$.

The following result is established in \cite[Theorem 1.6]{jt-foundational}, see also \cite[\S 2]{ellis} and the references in \cite{jt-foundational} for related results in the literature.
\begin{theorem}[Canonical disintegration]
Let $\Gamma$ be a discrete group.
Let $(X,\mu,T)$ and $(Y,\nu,S)$ be $\OpProbAlgG$-dynamical systems, and let $\pi \colon X \to Y$ be a $\OpProbAlgG$-factor map.
Then there is a unique Radon probability measure $\mu_y$ on $\Stone(X)$ for each $y \in \Stone(Y)$ which depends continuously on $y$ in the vague topology in the sense that $y \mapsto \int_{\Stone(X)} f\ d\mu_y$ is continuous for every $f$ in the space of continuous functions $C(\Stone(X))$, and such that
\begin{equation}\label{disint-form} \begin{split}
&\int_{\Stone(X)} f(x) g(\Stone(\pi)(x))\ d\mu_{\Stone(X)}(x)\\[2mm]
 &  = \int_{\Stone(Y)} \left(\int_{\Stone(X)} f\ d\mu_y\right) g\ d\mu_{\Stone(Y)}\end{split}
\end{equation}
for all $f \in C(\Stone(X))$, $g \in C(\Stone(Y))$.  Furthermore, for each $y \in \Stone(Y)$, $\mu_y$ is supported on the compact set $\Stone(\pi)^{-1}(\{y\})$, in the sense that $\mu_Y(E)=0$ whenever $E$ is a measurable set disjoint from $\Stone(\pi)^{-1}(\{y\})$. (Note that this conclusion does \emph{not} require the fibers $\Stone(\pi)^{-1}(\{y\})$ to be measurable.)
Moreover, we have $\mu_{S^\gamma_{\Stone(Y)}(y)}=(T^\gamma_{\Stone(X)})_\#\mu_y$ for all $y\in \Stone(Y)$ and $\gamma\in\Gamma$.
\end{theorem}

We can use the canonical disintegration to define relatively independent products.
Let $(X,\mu,T)$ and $(Y,\nu,S)$ be $\OpProbAlgG$-dynamical systems and $\pi \colon X \to Y$ a $\OpProbAlgG$-factor map. Then we can define the $\CHProb_\Gamma$-dynamical system  \begin{equation*} \begin{split}&(\Stone(X)\times\Stone(X),\Baire(\Stone(X)\times\Stone(X)),\\ &  \quad \quad  \quad \quad \quad \quad\quad \quad \quad \ \  \mu_{\Stone(X)}\times_{\Stone(Y)}\mu_{\Stone(X)}, T_{\Stone(X)}\times T_{\Stone(X)})
\end{split}
\end{equation*}
 as follows.
Recall that $\Baire(\Stone(X)\times\Stone(X))=\Baire(\Stone(X))\times \Baire(\Stone(X)),$ see e.g.~\cite[Lemma 2.1]{jt19}. This is the main reason for adopting the ``Baire-centric''  perspective, see \cite{jt-foundational} for further illustration of this perspective.

We define the relatively independent product measure
\begin{equation*}
\mu_{\Stone(X)}\times_{\Stone(Y)}\mu_{\Stone(X)}(E)=\int_{\Stone(Y)}\mu_{y}\times\mu_{y}(E)d\nu_{\Stone(Y)}\end{equation*}
for $E\in \Baire(\Stone(X)\times\Stone(X))$.

Finally, we define the product action $T_{\Stone(X)}\times T_{\Stone(X)}$ by
\begin{align*}
    &\Gamma\to \Aut(\Stone(X)\times\Stone(X),\Baire(\Stone(X)\times\Stone(X)),\mu_{\Stone(X)}\times_{\Stone(Y)}\mu_{\Stone(X)}) \\[2mm]
    &\gamma\mapsto (T_{\Stone(X)}\times T_{\Stone(X)})^\gamma(x,y):=T^\gamma_{\Stone(X)}(x)\times T^\gamma_{\Stone(X)}(y).
\end{align*}
This canonical relatively independent product has exactly the same properties as its classical counterpart for standard Borel spaces and countable group actions as for example recorded in \cite[Section 5.5]{furstenberg2014recurrence}, see \cite[\S 8]{jt-foundational}. In particular, we have
\begin{equation}\label{eq-indepprod}\begin{split}
   & \int_{\Stone(Y)}\E(f|Y)\E(g|Y)d\mu_{\Stone(Y)}\\[2mm]
    & =\int_{\Stone(X)\times\Stone(X)} f\times g \ d\mu_{\Stone(X)}\times_{\Stone(Y)}\mu_{\Stone(X)}\end{split}
\end{equation}

\section{An ergodic Roth theorem for uncountable amenable groups, and an application}\label{sec-convergence}

In this section we prove Theorem \ref{thm-uncountableroth} and derive from it a combinatorial application.
In Section \ref{subsec1}, we describe the conditional Kronecker factors which control the convergence of the non-conventional ergodic averages as occurring in Theorem \ref{thm-uncountableroth}.
We verify in Lemma \ref{lem-characteristic} that these factors are characteristic.
We provide the necessary versions of the mean ergodic theorem (Theorem \ref{thm-neumann}) and the van der Corput lemma (Lemma \ref{lem-corput}) for uncountable group actions needed to prove Lemma \ref{lem-characteristic}.
We then prove Theorem \ref{thm-uncountableroth} along the lines of the proof in \cite{bergelson1997roth}.
Finally, we apply Theorem \ref{thm-uncountableroth} in Section \ref{subsec2} to find triangular configurations in dense subsets of arbitrary amenable groups. To this end, we prove a version of the correspondence principle of Furstenberg for arbitrary amenable groups.

\subsection{Characteristic factors and non-conventional averages}\label{subsec1}
The Kronecker factor is characteristic for the non-conventional ergodic averages in the ergodic Roth theorem for ergodic $\mathbb{Z}$-actions, see \cite[\S 3]{furstenberg1977ergodic}.
When considering two commuting actions $S,T$ of an amenable group, Bergelson, McCutcheon and Zhang \cite[\S 4]{bergelson1997roth} identified the factors controlling the convergence of the non-conventional ergodic averages occuring in the ergodic Roth theorem as compact extensions of the corresponding $T$ and $ST$-systems over their common $S$-invariant factor (see below for details).
A careful Furstenberg-Zimmer structural analysis of compact extensions of two systems with respect to a common factor is carried out in \cite[\S 2,3]{bergelson1997roth}. This is applied in {\cite[\S 4,5]{bergelson1997roth}} to establish the existence and positivity of the limit in the amenable ergodic Roth theorem for countable groups.
Since we can use \cite[Theorem 1.1]{zorin2016norm} of Zorin-Kranich to get the existence of the limit in our ergodic Roth theorem for uncountable amenable groups (Theorem \ref{thm-uncountableroth}), we can largely avoid adapting the arguments in \cite[\S 2,3]{bergelson1997roth} to an uncountable/inseparable setting and focus on the necessary modifications needed to establish positivity. We remark that the Furstenberg-Zimmer type structural analysis in \cite[\S 2-4]{bergelson1997roth} can be fully adapted, following the uncountable Furstenberg-Zimmer structure theory developed by the fourth author in \cite{jamneshan2019fz}, to obtain also a new ergodic theoretic proof of Zorin-Kranich's \cite[Theorem 1.1]{zorin2016norm} in the case of two commuting actions of an arbitrary discrete amenable group on an arbitrary space.

For the remainder of this section fix an arbitrary discrete amenable group $\Gamma$.
A F\o lner net for $\Gamma$ is denoted by $(\Phi_\alpha)_{\alpha\in A}$ and recall our standing assumption that all F\o lner nets are understood to be left F\o lner nets.
We start by collecting two well-known results.

\begin{theorem}[Mean ergodic theorem]\label{thm-neumann}
Let $(X,\mu,T)$ be a $\OpProbAlgG$-system and $f\in L^2(X)$. Then we have
\begin{equation*}
\lim_{\alpha\in A} \frac{1}{|\Phi_\alpha|}\sum_{\gamma\in \Phi_\alpha} T^\gamma_{\Stone(X)}(f)=\E(f|\Inv(X,\mu,T))
\end{equation*}
in $L^2(X)$.
\end{theorem}

\begin{proof}
E.g., see \cite[Theorem 5.7]{paterson}.
\end{proof}

For the sake of completeness, we give a proof of van der Corput's lemma for uncountable amenable groups  by adapting the arguments in \cite{bergelson1997roth}.

\begin{lemma}\label{lem-corput}
Let $\{f_{\gamma}:\gamma\in \Gamma\}$ be a subset of a (not necessarily separable) Hilbert space $\mathcal{H}$ such that $\sup_{\gamma\in \Gamma} \|f_\gamma\|_{\mathcal{H}}<\infty$.
If
\begin{equation}\label{eq:corput}
\lim_{\alpha \in A} \frac{1}{|\Phi_\alpha|^2} \bigg(\limsup_{\beta\in A} \frac{1}{|\Phi_\beta|} \sum_{\gamma\in \Phi_\beta}\sum_{\eta,\rho\in \Phi_\alpha} \langle f_{\eta \gamma}, f_{\rho\gamma}\rangle_{\mathcal{H}}\bigg)=0,
\end{equation}
then
\begin{equation*}
\lim_{\alpha\in A} \bigg\|\frac{1}{|\Phi_\alpha|}\sum_{\gamma\in \Phi_\alpha} f_\gamma\bigg\|_{\mathcal{H}}=0.
\end{equation*}
\end{lemma}
The $\limsup$ in \eqref{eq:corput} can be defined in a various of equivalent ways. For example, for a net $(x_\alpha)_{\alpha\in A}$ of real numbers, where $(A,\leq)$ is a directed set (that is, a partially ordered set with the property that for every pair $a,b\in A$ there is $c\in A$ such that $a,b\leq c$). We can define $\limsup_{\alpha\in A} x_\alpha:= \inf_{\alpha\in A} \sup_{\beta\geq \alpha} x_\beta$, where $\{\beta\geq \alpha\}$ are the ``tail'' of the net.

\begin{proof}
Let us rewrite
\begin{equation}
\label{eq:1term}\begin{split}
\frac{1}{|\Phi_\beta|} \sum_{\gamma\in \Phi_\beta} f_\gamma =  & \Bigg( \frac{1}{|\Phi_\beta|} \sum_{\gamma\in \Phi_\beta} \frac{1}{|\Phi_\alpha|}\sum_{\omega\in \Phi_\alpha} f_{\omega \gamma} \Bigg)  \\[2mm] & + \Bigg ( \frac{1}{|\Phi_\beta|} \sum_{\gamma\in \Phi_\beta} f_\gamma - \frac{1}{|\Phi_\beta|} \sum_{\gamma\in \Phi_\beta} \frac{1}{|\Phi_\alpha|} \sum_{\omega\in \Phi_\alpha} f_{\omega \gamma}\Bigg ). \end{split}
\end{equation}
Note that for the second term of the right-hand side of \eqref{eq:1term},
\begin{equation}\label{eq:2term}\begin{split}&
    \bigg\|\frac{1}{|\Phi_\beta|} \sum_{\gamma\in \Phi_\beta} f_\gamma - \frac{1}{|\Phi_\beta|} \sum_{\gamma\in \Phi_\beta} \frac{1}{|\Phi_\alpha|} \sum_{\omega\in \Phi_\alpha} f_{\omega \gamma}\bigg\|_{\mathcal{H}}\\[2mm]
     & = \bigg\|\frac{1}{|\Phi_\beta|} (\sum_{\gamma\in \Phi_\beta} f_\gamma - \frac{1}{|\Phi_\alpha|}  \sum_{\omega\in \Phi_\alpha} \sum_{\gamma\in \omega \Phi_\beta} f_{\gamma})\bigg\|_{\mathcal{H}}.\end{split}
\end{equation}
Since $\sup_{\omega\in \Phi_\alpha} \frac{|\Phi_\beta \Delta \omega \Phi_\beta|}{|\Phi_\beta|}\to 0$ by the F\o lner property (see Appendix \ref{sec-amsyn}), it follows that the right-hand side of \eqref{eq:2term}  converges to $0$ (in $\beta$).
Meanwhile, for the first term in the right-hand side of \eqref{eq:1term}, the Cauchy-Schwarz inequality yields
\begin{equation}\label{eq:conclusion}\begin{split}
\bigg\|\frac{1}{|\Phi_\beta|} \sum_{\gamma\in \Phi_\beta} \frac{1}{|\Phi_\alpha|}\sum_{\omega\in \Phi_\alpha} f_{\omega \gamma}\bigg\|_{\mathcal{H}}^2 &\leq \frac{1}{|\Phi_\beta|} \sum_{\gamma\in \Phi_\beta} \bigg\| \frac{1}{|\Phi_\alpha|} \sum_{\omega \in \Phi_\alpha} f_{\omega \gamma}\bigg\|_{\mathcal{H}}^2 \\[2mm] &= \frac{1}{|\Phi_\beta|}\sum_{\gamma\in \Phi_\beta} \frac{1}{|\Phi_\alpha|^2}\sum_{\omega,\rho\in \Phi_\alpha} \langle f_{\omega \gamma}, f_{\rho \gamma} \rangle_{\mathcal{H}}.  \end{split}
\end{equation}
Now taking the norm on both sides of \eqref{eq:1term}, using the triangle inequality to separate the two terms on its right-hand side, choosing then first $\beta$ and second $\alpha$ large enough, the claim follows from the hypothesis \eqref{eq:corput}.
\end{proof}

We introduce the important notion of compact extensions.

\begin{definition}[Compact extension]\label{def-compactext}
An abstract $\OpProbAlgG$-extension $$\pi\colon (X,\mu,T) \to (Y,\nu,S)$$ is called a \emph{$\OpProbAlgG$-compact extension} if
$L^2(X)$ is the closure of the union of all finitely generated, closed, and $T$-invariant $L^\infty(Y)$-submodules of $L^2(X)$.
\end{definition}

The interested reader is referred to \cite[\S 4]{jamneshan2019fz} for different descriptions of abstract compact extension. One of these descriptions in terms of conditionally almost periodic functions is used in the proof of our uncountable amenable ergodic Roth theorem below and is derived in Appendix \ref{sec-canalysis}. For a more extensive treatment we refer to \cite[\S 4]{jamneshan2019fz}.

Given an arbitrary $\OpProbAlgG$-extension $\pi:(X,\mu,T)\to (Y,\nu,S)$, there is a largest compact extension of $Y$ below $X$.
More precisely, let $H$ be the closure of the union of all finitely generated, closed, and $T$-invariant $L^\infty(Y)$-submodules of $L^2(X)$. Since $H$ is closed under complex conjugation and multiplication, we can identify $H$ with a von Neumann subalgebra of $L^\infty(X)$. Now using the well known duality between categories of commutative von Neumann algebras and probability algebras, we can identify with $H$ a $\OpProbAlgG$-factor $(Z,\lambda, R)$ of $(X,\mu,T)$ which is an $\OpProbAlgG$-extension of $(Y,\nu,S)$, see \cite{jt-foundational} and \cite[\S 2]{jamneshan2019fz} for a reference.

For the remainder of this section,  we fix a $\OpProbAlgG$-Roth dynamical system $(X,\mu,S,T)$.
Then $(X, \mu, ST)$ is a $\OpProbAlgG$-dynamical system where the abstract action $ST: \Gamma\to \Aut(X,\mu)$ is defined by $\gamma \mapsto S^\gamma\circ T^\gamma$. We write $S^\gamma T^\gamma=S^\gamma \circ T^\gamma$.
It follows from the commutativity of $T,S$ that $(\Inv(X,\mu,S),\mu,T)$ is a $\OpProbAlgG$-factor of $(X,\mu, T)$ and $(X,\mu, ST)$ respectively.
In order to lighten the notation, we denote by $Y=\Inv(X,\mu,S)$.
Let $H$ be the closure of the union of all finitely generated, closed, and $T$-invariant $L^\infty(Y)$-submodules of $L^2(X)$.
We call the $\OpProbAlgG$-factor $(Z_T,\mu_{Z_T},T_{Z_T})$ of $(X,\mu,T)$ identified with $H$ the \emph{conditional Kronecker factor} of  $(X,\mu,T)$.
Similarly, we define the conditional Kronecker factor of $(Z_{ST},\mu_{Z_{ST}},T_{Z_{ST}})$ of $(X, \mu, ST)$ relative to the invariant factor $Y$.

Going forward, we work with the canonical models of the above systems and their corresponding $\CHProb_\Gamma$-factor relations.
For example, the canonical model of $(X,\mu,S,T)$ is denoted by $(\Stone(X),\Baire(\Stone(X)),\mu_{\Stone(X)},S_{\Stone(X)},T_{\Stone(X)})$ and we have the $\CHProb_\Gamma$-factor map \begin{equation*}\begin{split}
&(\Stone(Z_{ST}),\Baire(\Stone(Z_{ST})),\mu_{\Stone(Z_{ST})},(ST)_{\Stone(Z_{ST})}) \\[2mm] & \to (\Stone(Y),\Baire(Y),\mu_{\Stone(Y)},T_{\Stone(Y)}).
\end{split}
\end{equation*}

\begin{lemma}\label{lem-characteristic}
Suppose that $f,g\in L^\infty(X)$ with $f \perp L^2(Z_T)$ or $g \perp L^2(Z_{ST})$, where we view $L^2(Z_T)$ and  $L^2(Z_{ST})$ as subspaces of $L^2(X)$.
Then
\begin{equation}\label{eq_Thm43}
\lim_{\alpha \in A} \frac{1}{|\Phi_\alpha|} \sum_{\gamma\in \Phi_\alpha} T^\gamma(f) S^\gamma T^\gamma(g)=0
\end{equation}
in $L^2(X)$.
\end{lemma}

We adapt the arguments in \cite[Theorem 4.3]{bergelson1997roth}, we include the details for completeness.

\begin{proof}
It is enough to consider the case $f \perp L^2(Z_T)$.
The proof for the case where {$g \perp L^2(Z_{ST})$} is similar.
Then
\begin{align*}
&\frac{1}{|\Phi_\alpha|} \sum_{\gamma\in \Phi_\alpha} \langle u_{\zeta\gamma},u_{\theta\gamma}\rangle_{L^2(X)}\\[2mm]
&=\frac{1}{|\Phi_\alpha|} \sum_{\gamma\in \Phi_\alpha} \int  T^{\zeta\gamma}(f)S^{\zeta\gamma}T^{\zeta\gamma}(g)\overline{ T^{\theta\gamma}(f)S^{\theta\gamma}T^{\theta\gamma}(g)}\mathrm{d}\mu_{\Stone(X)} \\[2mm]
&=\frac{1}{|\Phi_\alpha|} \sum_{\gamma\in \Phi_\alpha} \int T^{\zeta}(f)\overline{T^{\theta}(f)} S^{\zeta\gamma}T^{\zeta}(g)\overline{S^{\theta\gamma}T^{\theta}(g)}\mathrm{d}\mu_{\Stone(X)}\\[2mm]
&=\int T^{\zeta}(f)\overline{T^{\theta}(f)}\left( \frac{1}{|\Phi_\alpha|} \sum_{\gamma\in \Phi_\alpha} S^{\zeta}T^{\zeta}  S^{\gamma}(g)\overline{S^{\theta}T^{\theta}S^\gamma(g)}\right)\mathrm{d}\mu_{\Stone(X)}.
\end{align*}
The second equality is  due to $T^\gamma$ being measure preserving.
 By Theorem \ref{thm-neumann} and orthogonal decomposition, we have
 \begin{align*}
 &\frac{1}{|\Phi_\alpha|} \sum_{\gamma\in \Phi_\alpha} \langle u_{\zeta\gamma},u_{\theta \gamma}\rangle_{L^2(X)}\\[2mm]
&\quad \to  \int T^\zeta(f) T^\theta(\overline{f})\E(S^\zeta T^\zeta(g) S^\theta T^\theta( \overline{g})\mid Y)\mathrm{d} \mu_{\Stone(X)}\\[2mm]
& \quad \quad  \ =\int \E(T^\zeta(f) T^\theta(\overline{f})\mid Y)\E(S^\zeta T^\zeta(g) S^\theta T^\theta(\overline{g})\mid Y)\mathrm{d}\mu_{\Stone(Y)}.
\end{align*}
By \eqref{eq-indepprod} this equals
\begin{equation*} \int T^{\zeta}\times S^\zeta T^\zeta(f\times g)T^{\theta}\times S^{\theta}T^{\theta}(\overline{f\times g})\,d\mu_{\Stone(X)}\times_{\Stone(Y)}\mu_{\Stone(X)}\\
=:a_{\zeta,\theta}.\end{equation*}
Note that by Lemma \ref{lem-corput},
to conclude the proof it suffices to show that
\begin{equation*} \lim_{\alpha \in A}\frac{1}{|\Phi_{\alpha}|^2}\sum_{\zeta,\theta\in \Phi_{\alpha}}a_{\zeta,\theta} =0.\end{equation*}
Observe that the left hand-side of the last display equals
\begin{equation*}
\lim_{\alpha\in A}\|\frac{1}{|\Phi_\alpha|}\sum_{\gamma\in\Phi_\alpha} T^\gamma_{\Stone(X)}\times S^\gamma_{\Stone(X)} T^\gamma_{\Stone(X)}(f\times g)\|_{L^2(X)}^2.
\end{equation*}
This equals zero   by  another application of Theorem \ref{thm-neumann}
since $f\times g$ is orthogonal to the invariant factor of the relatively independent product with respect to the action of $T_{\Stone(X)}\times S_{\Stone(X)}T_{\Stone(X)}$. The proof is complete.
\end{proof}

We now prove our ergodic Roth theorem for uncountable amenable groups.

\begin{proof}[Proof of Theorem~\ref{thm-uncountableroth}] The existence of the limit \eqref{mainlimit} follows from~\cite[Theorem~1.1]{zorin2016norm}.
It remains to  establish positivity of \eqref{mainlimit}. More precisely, for every {$E\in \Baire(\Stone(X))$} with $\mu_{\Stone(X)}(E)>0$ we will show that
\begin{equation}\label{eq:limit}
\lim_{\alpha\in A} \frac{1}{|\Phi_\alpha|} \sum_{\gamma\in \Phi_\alpha} \mu_{\Stone(X)}(E\cap  T^\gamma_{\Stone(X)}(E)\cap S_{\Stone(X)}^\gamma T_{\Stone(X)}^\gamma(E))>0.
\end{equation}
By Lemma \ref{lem-characteristic} and an orthogonal decomposition, it is enough to show that
\begin{equation*}
\lim_{\alpha\in A} \frac{1}{|\Phi_\alpha|} \sum_{\gamma\in \Phi_\alpha} \int_{\Stone(X)} 1_E\E(1_E |Z_T) \E(1_E |Z_{ST}) d\mu_{\Stone(X)} >0.
\end{equation*}
Since $\mu_{\Stone(X)}(E)>0$, it holds that $\E(1_E|Z_T)\E(1_E|Z_{ST})>0$   on $E$.
Thus,
there exist $r>0$ and $E'\in\Baire(\Stone(X))$ with $E'\subset E$ and $\mu_{\Stone(X)}(E')>0$ such that  \begin{equation*} \E(1_E|Z_T)\E(1_E|Z_{ST})>r \quad \text{  on } E'.\end{equation*}  Recall that we denote by $Y$ the invariant factor of $(X,\mu,S)$.
Since \begin{equation*} 0<\mu_{\Stone(X)}(E')=\int_{\Stone(Y)} \E(1_{E'}|Y)d\mu_{\Stone(Y)},\end{equation*}  we find $t>0$ and $F\in \Baire(\Stone(Y))$ with $\mu_{\Stone(Y)}(F)>0$ such that $\E(1_{E'}|Y)>t$ on $F$. In particular, we obtain
\begin{equation*} \E(1_E \E(1_E|Z_T) \E(1_E|Z_{ST})|Y) \geq r \E(1_{E'}|Y) > rt \quad \text{  on } F.
\end{equation*}
Furthermore, by the definition of a compact extension, there exist sequences $(f_n)$ and $(g_n)$ such that each $f_n$  is contained in a $T$-invariant  finitely generated $L^\infty(Y)$-submodule of $L^2(X)$, each $g_n$ is contained in an $(ST)$-invariant finitely generated $L^\infty(Y)$-submodule of $L^2(X)$, and  \begin{equation*} \|\E(1_E|Z_T)-f_n\|_{L^2(X)}\to 0  \quad \textup{and}
\quad \|\E(1_E|Z_{ST})-g_n\|_{L^2(X)}\to 0.\end{equation*}
We can and will assume that $f_n$ and $g_n$ are bounded for all $n$.
Indeed, we can define $f_{n,m}= f_n 1_{\E(|f_n|^2|Y)\leq m}$.
Then $f_{n,m}$ is an element of the same $L^\infty(Y)$-submodule as $f_n$ for each $m$ and every $n$ and
the diagonal sequence $(f_{n,n})$ approximates $\E(1_E|Z_T)$ in $L^2(X)$. Similarly for $(g_n)$.

Using the identity
\begin{equation*} \|u\|_{L^2(X)}^2=\int_{\Stone(Y)} \E(|u|^2|Y)d\mu_{\Stone(Y)}\end{equation*}  and passing to a subsequence if necessary, we have that
\begin{equation*}  \E(|\E(1_E|Z_T)-f_n|^2|Y)\to 0  \quad \textup{and} \quad  \E(|\E(1_E|Z_{ST})-g_n|^2|Y)\to 0 \quad \text{a.s.} \end{equation*}
By Egorov's theorem, we find $F'\in \Baire(\Stone(Y))$ with $F'\subset F$ and $\mu_{\Stone(Y)}(F')>0$ such that $2\mu_{\Stone(Y)}(F')>\mu_{\Stone(Y)}(F)$ and the following two hold \begin{align*}
   \E(|\E(1_E|Z_T)-f_n|^2|Y)\to 0 \quad &\textup{a.s. uniformly on $F'$}  \\[2mm]
    \quad \E(|\E(1_E|Z_{ST})-g_n|^2|Y)\to 0 \quad   &\textup{a.s. uniformly on $F'$}
\end{align*}
 In particular, for any $\ep>0$ we find some $f$ and $g$ in  a $T$- and  $(ST)$-invariant finitely generated $L^\infty(Y)$-submodule of $L^2(X)$, respectively, such that
\begin{equation}
\label{eqnapprox1}\begin{split}
    \E(|\E(1_E|Z_T)-f|^2|Y)<\ep &\quad  \textup{on $F'$} \\[2mm]   \E(|\E(1_E|Z_{ST})-g|^2|Y)<\ep &\quad \textup{on $F'$}.
    \end{split}
\end{equation}
Since both $f,g$ lie in finitely generated, closed, and $T$- and $(ST)$-invariant $L^\infty(Y)$-submodules of $L^2(X)$, respectively, it follows from (ii)' implies (iii)' in \cite[Theorem 4.1]{jamneshan2019fz} that there are $h_1,\ldots,h_l\in L^2(X)$ such that for each $\gamma\in\Gamma$,  the following two hold
\begin{equation}
\label{eqnapprox2} \begin{split}
    \E(|T_{\Stone(X)}^\gamma(f)- h_{N^T_\gamma}|^2|Y)<\ep  &\quad \textup{on $F'$}, \\[2mm] \E(|S_{\Stone(X)}^\gamma T_{\Stone(X)}^\gamma(g)- h_{N^{ST}_\gamma}|^2|Y)<\ep &\quad  \textup{on $F'$}, \end{split}
\end{equation}
where
\begin{equation*} N^T_\gamma=\sum_{m=1}^l m 1_{C_{m,\gamma}},\quad N^{ST}_\gamma=\sum_{m=1}^l m 1_{D_{m,\gamma}},\end{equation*}     and $(C_{m,\gamma})$ and $(D_{m,\gamma})$ are defined as follows.

 Let $
\tilde{C}_{m,\gamma}=\{\E(|T_{\Stone(X)}^\gamma(f)- h_m|^2|Y)<\ep\}$
 for $m=1,\ldots,l$, and set $C_{1,\gamma}=\tilde{C}_{1,\gamma}$ and $C_{m,\gamma}=\tilde{C}_{m,\gamma}\backslash \bigcup_{m'=1}^{m-1} \tilde{C}_{m',\gamma}$ for $m=2,\ldots,l$. Similarly define $D_{m,\gamma}$, $m=1,\ldots,l$ with $T_{\Stone(X)}^\gamma(f)$ replaced by $S_{\Stone(X)}^\gamma T_{\Stone(X)}^\gamma(g)$. A self-contained proof of the above result in \cite{jamneshan2019fz} tailored to our setting is Lemma \ref{lem-heineborel} in Appendix \ref{sec-canalysis}.

Next we show the following claim.

\vspace*{4pt}\noindent{\bf Claim.}
  \emph{Let $M>0$.  For any  $0<\delta<1$
there exist $\tilde{\gamma}_1,\ldots,\tilde{\gamma}_M\in \Gamma$ such that}
\begin{itemize}
\item[(i)]\footnote{\emph{We identify a subset $Y'$ of the factor $Y$ with the subset $\pi^{-1}(Y')$ of $X$, where $\pi:(X,\mu,S)\to Y$ is the factor map.}} $\mu_{\Stone(X)}(F' \cap  (T_{\Stone(X)}^{\tilde{\gamma}_1})^{-1} F'\cap \ldots \cap( T_{\Stone(X)}^{\tilde{\gamma}_M})^{-1} F')>(\delta\mu_{\Stone(X)}(F'))^{2^M}$,
\item[(ii)] $\tilde{\gamma}_i^{-1}\tilde{\gamma}_j\in \Phi_{\alpha}$ \emph{for some $\alpha\in A$ whenever} $1\leq i<j\leq M$.
\end{itemize}

\vspace*{4pt}\noindent
{\em Proof of claim.} Observe that
by Theorem \ref{thm-neumann},
\begin{align}
\lim_{\alpha \in A} \frac{1}{|\Phi_\alpha|}\sum_{\gamma\in\Phi_\alpha} \mu_{\Stone(X)}(F' \cap (T^\gamma_{\Stone(X)})^{-1}(F')) = \langle 1_{F'}, \E[1_{F'}|Z_T] \rangle_{L^2(X)}
\end{align}
This equals $\| \E[1_{F'}|Z_T]] \|^2_{L^2(x)}$, which is bounded from below by
\begin{equation*} \| \E[1_{F'}|Z_T]] \|^2_{L^1(x)} = \mu_{\Stone(X)}(F')^2.\end{equation*}
By the pigeonhole principle, for every $0< \delta <1$ there exist $\alpha_1\in A$ and $\gamma_1\in \Phi_{\alpha_1}$ such that \begin{equation*} \mu_{\Stone(X)}(F'\cap (T_{\Stone(X)}^{\gamma_1})^{-1}(F'))> \delta (\mu_{\Stone(X)}(F'))^2.\end{equation*}

For a given F\o lner net $(\Phi_\alpha)_{\alpha\in A}$ and
 any finite family $\gamma_1,\ldots, \gamma_k \in \Gamma$,  we have that
$(\Phi_\alpha\cap \gamma_1 \Phi_\alpha \cap \ldots \cap \gamma_k \Phi_\alpha)_{\alpha\in A}$ is also a F\o lner net. Thus, we may iterate the previous argument.

Let $F_1=F'\cap (T_{\Stone(X)}^{\gamma_1})^{-1}(F')$. Repeating the previous argument, where we replace $F'$ by $F_1$, we find $\alpha_2\in A$ and $\gamma_2\in \Phi_{\alpha_2} \cap \gamma^{-1}_1\Phi_{\alpha_2}$ such that \begin{equation*} \mu_{\Stone(X)}(F_1\cap (T_{\Stone(X)}^{\gamma_2})^{-1}(F_1))>\delta (\mu_{\Stone(X)}(F_1))^2.\end{equation*}  This implies
\begin{equation*}\begin{split}
&\mu_{\Stone(X)}(F'\cap (T^{\gamma_1}_{\Stone(X)})^{-1}(F')\cap (T^{\gamma_2}_{\Stone(X)})^{-1}(F')\cap (T^{\gamma_1\gamma_2}_{\Stone(X)})^{-1}(F'))\\[2mm] &
>\delta^3 \mu_{\Stone(X)}(F')^4.
\end{split}
\end{equation*}
 Set  $F_2=F_1\cap (T_{\Stone(X)}^{\gamma_2})^{-1}(F_1)$. Choose $\alpha_3\in A$ and $\gamma_3\in \Phi_{\alpha_3}\cap \gamma_1^{-1}\Phi_{\alpha_3} \cap \gamma_2^{-1}\Phi_{\alpha_3} \cap (\gamma_1\gamma_2)^{-1}\Phi_{\alpha_3}$ such that \begin{equation*} \mu_{\Stone(X)}(F_2 \cap (T^{\gamma_3}_{\Stone(X)}(F_2))>\delta \mu_{\Stone(X)}(F_2)^2.\end{equation*}
After $M$ iterations we  obtain  $\tilde{\gamma}_1,\ldots,\tilde{\gamma}_M$ with the properties (i) and (ii), where  $\tilde{\gamma}_i=\gamma_1\gamma_2\cdots \gamma_i$ for $1\leq i\leq M$. This completes the proof of the claim. \qed

Now we choose $M=l^2+1$. By the above claim we find  $\tilde{\gamma}_1,\ldots,\tilde{\gamma}_M$ with the properties (i) and (ii).  By the pigeonhole principle, one can construct measurable functions $I,J:\Stone(Y)\to \{1,\ldots,M\}$ such that $I(y)<J(y)$,   $N^T_{\tilde{\gamma}_{I(y)}}(y)=N^T_{\tilde{\gamma}_{J(y)}}(y)$, and  $N^{ST}_{\tilde{\gamma}_{I(y)}}(y)=N^{ST}_{\tilde{\gamma}_{J(y)}}(y)$.
 Using \eqref{eqnapprox2} and  the triangle inequality, we obtain
\begin{align*}
       &\E(|T_{\Stone(X)}^{\tilde{\gamma}_I}(f)- T_{\Stone(X)}^{\tilde{\gamma}_J}(f) |^2|Y)<2\ep \quad \textup{on $F'$ \quad{and}}\\[2mm] &\E(|S_{\Stone(X)}^{\tilde{\gamma}_I}T_{\Stone(X)}^{\tilde{\gamma}_I}(g)- S_{\Stone(X)}^{\tilde{\gamma}_J} T_{\Stone(X)}^{\tilde{\gamma}_J}(g) |^2|Y)<2\ep \quad \textup{on } F'.
\end{align*}
These inequalities imply
\begin{align}\label{eqnapprox3}
   & \E(|f-T^{\tilde{\gamma}^{-1}_{I}\tilde{\gamma}_{J}}_{\Stone(X)}(f)|^2|Y)\circ T^{\tilde{\gamma}^{-1}_I}< 2\ep \quad \text{on } F'.\\[2mm] \nonumber
&     \E(|g-S^{\tilde{\gamma}^{-1}_{I}\tilde{\gamma}_{J}}_{\Stone(X)}T^{\tilde{\gamma}^{-1}_{I}\tilde{\gamma}_{J}}_{\Stone(X)}(g)|^2|Y)\circ T^{\tilde{\gamma}^{-1}_I}_{\Stone(X)}< 2\ep \quad \text{on $F'$}.
 \end{align}
 In the second inequality we have used that $S$ acts trivially on $Y$.
 Let us set \begin{equation*} D:=F' \cap T_{\Stone(X)}^{\tilde{\gamma}_1} F'\cap \ldots \cap  T_{\Stone(X)}^{\tilde{\gamma}_M} F'.\end{equation*}
Now for $y \in D$, we have $(T_{\Stone(X)}^{\tilde{\gamma}_{I(y)}})^{-1}(y)\in F'$.
Therefore, by \eqref{eqnapprox1} we obtain
\begin{align} \label{eqnapprox4}
    &T_{\Stone(X)}^{\tilde{\gamma}_{I(y)}}(\E(|\E(1_E|Z_T)-f|^2|Y))(y)<\varepsilon  \\[2mm] \nonumber
    &T_{\Stone(X)}^{\tilde{\gamma}_{I(y)}}(\E(|\E(1_E|Z_{ST})-g|^2|Y))(y)<\ep.
\end{align}
 for almost every $y \in D$.
 Moreover, since   $(T_{\Stone(X)}^{\tilde{\gamma}_{J(y)}})^{-1}(y)\in F'$,  we have
\begin{equation}
\label{eqnapprox5}\begin{split}
&T_{\Stone(X)}^{\tilde{\gamma}_I}(\E(|T^{\tilde{\gamma}^{-1}_{I}\tilde{\gamma}_{J}}_{\Stone(X)}(\E(1_E|Z_T)-T^{\tilde{\gamma}^{-1}_{I}\tilde{\gamma}_{J}}_{\Stone(X)}(f)|^2|Y) \\&  = T^{\tilde{\gamma}_J}_{\Stone(X)}\E(|\E(1_E|Z_T)-f|^2|Y)<\ep \quad \text{on $D$.} \end{split}
\end{equation}

Applying the triangle inequality to \eqref{eqnapprox3}, \eqref{eqnapprox4}, and \eqref{eqnapprox5}, we obtain
\begin{equation*} T^{\tilde{\gamma}_I}_{\Stone(X)}\E(|\E(1_E|Z_T)-T_{\Stone(X)}^{\tilde{\gamma}_{I}^{-1}\tilde{\gamma}_J}(\E(1_E|Z_T))|^2|Y)<4\ep\quad \text{on $D$.}\end{equation*}
Similarly, we have
\begin{equation*} T^{\tilde{\gamma}_I}_{\Stone(X)}\E(|\E(1_E|Z_{ST})-T_{\Stone(X)}^{\tilde{\gamma}_{I}^{-1}\tilde{\gamma}_J}(\E(1_E|Z_{ST}))|^2|Y)<4\ep \quad \text{on $D$}\end{equation*}
Let $\tilde{\gamma}:=\tilde{\gamma}(y)=\tilde{\gamma}_{I(y)}^{-1}\tilde{\gamma}_{J(y)}$.
Setting $\ep=\frac{rt}{16}$, we then obtain
\begin{equation*}\begin{split}
&\E(1_E T_{\Stone(X)}^{\tilde{\gamma}}(\E(1_E|Z_T))S_{\Stone(X)}^{\tilde{\gamma}}T_{\Stone(X)}^{\tilde{\gamma}}(\E(1_E|Z_{ST}))|Y)\\[2mm]&
> \E(1_E \E(1_E|Z_T)\E(1_E|Z_{ST})|Y) - 8\ep>\frac{rt}{2} \quad \text{on }D.
\end{split}
\end{equation*}

Let $\eta=(\delta\mu_{\Stone(X)}(F'))^{2^M}$. We can choose $B\subset D$ with  $\mu_{\Stone(X)}(B)>\frac{\eta}{M^2}$ such that $\tilde{\gamma}_0=\tilde{\gamma}(y)$ is constant on $B$.
Then
\begin{equation*}
\int_{\Stone(X)} 1_E T_{\Stone(X)}^{\tilde{\gamma}_0}(\E(1_E|Z_T)) S_{\Stone(X)}^{\tilde{\gamma}_0} T_{\Stone(X)}^{\tilde{\gamma}_0}(\E(1_E|Z_{ST}))>\frac{rt\eta}{2M^2}>0.
\end{equation*}
We have that $\tilde{\gamma}_0\in \Phi_\alpha$ for some $\alpha\in A$.
It follows from Proposition \ref{prop-syndetic} that
{\small\begin{equation*}
G=\bigg\{\tilde{\gamma}_0\in\Gamma: \int_{\Stone(X)} 1_E T_{\Stone(X)}^{\tilde{\gamma}_0}(\E(1_E|Z_T))S_{\Stone(X)}^{\tilde{\gamma}_0} T_{\Stone(X)}^{\tilde{\gamma}_0}(\E(1_E|Z_{ST}))>\frac{rt\eta}{2M^2}\bigg\}
\end{equation*}}is syndetic, since neither of the quantities $\E(1_E|Z_T), \E(1_E|Z_{ST}), r,t,\eta,M$ depend on the choice of the F\o lner net.
It follows from Lemma \ref{lem-density-syndeticity} that $\BD_\Gamma(G)>0$.
Thus
\begin{equation*}\begin{split}
&\lim_{\alpha \in A} \frac{1}{|\Phi_\alpha|}\sum_{\gamma\in \Phi_\alpha} \int_{\Stone(X)} 1_E T_{\Stone(X)}^{\tilde{\gamma}_0}(\E(1_E|Z_T))S_{\Stone(X)}^{\tilde{\gamma}_0} T_{\Stone(X)}^{\tilde{\gamma}_0}(\E(1_E|Z_{ST}))\\[2mm]&
>\bar{d}_\Phi(G)\frac{rt\eta}{2M^2}>0,
\end{split}
\end{equation*}
where $\bar{d}_\Phi(G)=\limsup_{\alpha} \frac{|G\cap \Phi_\alpha|}{|\Phi_\alpha|}$. We therefore obtain \eqref{eq:limit}, as needed.
\end{proof}

From Theorem \ref{thm-uncountableroth} we deduce
  Corollary \ref{maincor}, about the largeness of the set of return times.

\begin{proof}[Proof of Corollary \ref{maincor}]
Assume towards a contradiction that for each $n\in \mathbb{N}$,
\begin{equation*}
G_n=\{\gamma\in \Gamma: \mu(E\wedge T^\gamma E\wedge S^\gamma T^{\gamma} E)>1/n\}
\end{equation*}
is not syndetic. Let $(\Phi_\alpha)_{\alpha\in A}$ be a F\o lner net.
Since $G_n$ is not syndetic there exists $\gamma^n_\alpha\in \Gamma\setminus \Phi_\alpha^{-1} G_n$ for each $\alpha\in A,n\in\mathbb{N}$.
We have that $(\Phi^n_\alpha)_{\alpha\in A}$ defined by $ \Phi^n_\alpha:=\Phi_\alpha \gamma_\alpha^n$ is another F\o lner net for each $n\in\N$ such that $\Phi^n_\alpha \cap G_n=\emptyset$ for all $\alpha\in A,n\in \mathbb{N}$.
By construction, for every $\alpha\in A$,
\begin{equation}\label{avg}
    \frac{1}{|\Phi^n_\alpha|}\sum_{\gamma\in \Phi^n_\alpha} \mu(E\wedge T^\gamma E\wedge S^\gamma T^{\gamma} E) \leq \frac{1}{n}.
\end{equation}
By \cite[Theorem 1.1 (2)]{zorin2016norm}, the limit
\begin{equation*}  \lim_{\alpha \in A}  \frac{1}{|\Phi^n_\alpha|}\sum_{\gamma\in \Phi^n_\alpha} \mu(E\wedge T^\gamma E\wedge S^\gamma T^{\gamma} E)
\end{equation*}
exists and is independent of the choice of the F\o lner net. Moreover, by \eqref{avg}, this limit is less than any $\ep>0$, hence it must be zero. However, this contradicts Theorem \ref{thm-uncountableroth}.
\end{proof}

\subsection{Triangular patterns in $\Gamma\times\Gamma$ }\label{subsec2}

In this section we   employ Theorem~\ref{thm-uncountableroth} and Corollary~\ref{maincor} to prove Theorem \ref{prop-triangular}.
We proceed similarly as in  \cite{bergelson1997roth}.  By \cite[Proposition 0.16 (5)]{paterson}, $\Gamma\times\Gamma$ is a discrete amenable group if $\Gamma$ is so as well.
We will need a correspondence principle analogous to \cite[Proposition 6.2]{bergelson1997roth}.
In \cite[Theorem 2.1]{bergelson-schur}, Bergelson and McCutcheon establish a correspondence principle for countable amenable semigroups.  We adapt this proof for uncountable discrete amenable groups.

Let $\Omega=\{0,1\}^{\Gamma\times\Gamma}$. Then $\Omega$ is a totally disconnected compact Hausdorff space, thus a Stone space (see Appendix A). An element $\omega\in \Omega$ corresponds uniquely to a subset of $\Gamma\times\Gamma$.  We define an action of $\Gamma\times \Gamma$ on $\Omega$ as follows. First let $S:\Gamma\to \Aut(\Omega)$ be defined by $S^\gamma(\omega)(\theta,\zeta):=(\theta \gamma,\zeta)$ and $T:\Gamma\to \Aut(\Omega)$ be defined by $T^\gamma(\omega)(\theta,\zeta):= (\theta ,\zeta \gamma)$.
Since $T,S$ are commuting, so we can define
\begin{equation*}U:\Gamma\times\Gamma\to \Aut(\Omega), \ \ U(\theta,\gamma):=S(\theta)\circ T(\gamma).\end{equation*}
\begin{lemma}[An uncountable Furstenberg correspondence principle] \label{lem-correspondence}
Fix an invariant mean $m:\ell^\infty(\Gamma\times\Gamma)\to \R$ and let $\Lambda\subset \Gamma\times \Gamma$ be such that $m(1_\Lambda)>0$.
Let~$X$ be the $U$-orbit closure of $1_\Lambda$ in $\Omega$, that is, $X:=\overline{\{S^{\theta}T^{\gamma}(1_{\Lambda}):\theta,\gamma\in\Gamma\}}$.
Then there exits a $\CHProb_\Gamma$-Roth dynamical system $(X,\Baire(X),\mu,S,T)$ such that $\mu(\{\omega\in X: \omega(e,e)=1\})>0$ where $e$ is the identity element of the group $\Gamma$ and $\omega(e,e)$ is the evaluation of $\omega\in \Omega$ at the entry $(e,e)$.
\end{lemma}

\begin{proof}
The collection $\mathcal{O}$  of cylinder sets
\begin{equation}\label{eq:cylinder}
\{x\in X: x(\gamma_1)=a_1,\ldots, x(\gamma_k)=a_k\}
\end{equation}
(where $k\in\N$, $\gamma_i\in \Gamma\times\Gamma$, $a_i\in\{0,1\}$, $1\leq i\leq k$) is a clopen base of the topology of $X$.  Let $\mathcal{A}$ be the Boolean algebra generated by $\mathcal{O}$ and $\mathcal{X}$ be the corresponding $\sigma$-algebra. By the  Stone-Weierstra\ss\  theorem, we have $\mathcal{X}=\Baire(X)$.
For a cylinder set $D$ of the form \eqref{eq:cylinder}, define
\begin{equation*}
\mu(D):=m(1_{\gamma_1^{-1} \Lambda_1}\cdot \ldots \cdot 1_{\gamma_k^{-1} \Lambda_k}),
\end{equation*}
 where $\Lambda_i=\Lambda$ if $a_i=1$ and $\Lambda_i=\Lambda^c$ if $a_i=0$, $1\leq i\leq k$.
 By compactness, $\mu$ is a premeasure on $\mathcal{A}$ and thus can be extended to a Baire probability measure on $\mathcal{X}$ by the Carath\'eodory extension theorem.
 Any Baire probability measure on a compact Hausdorff space is Radon  (e.g., see \cite[Proposition 4.2(iii)]{jt-foundational}).
 By construction, $\mu$ is $T$- and $S$-invariant and satisfies
\begin{equation*}\mu(\{\omega\in X: \omega(e,e)=1\})=m(1_\Lambda)>0.\end{equation*}
\end{proof}

\begin{proof}[Proof of Proposition \ref{prop-triangular}]
By Lemma \ref{lem-correspondence}, there exists a $U$-invariant Baire probability measure $\mu$ on  $X:=\overline{\{S^{\theta}T^{\gamma}1_{E}:\theta,\gamma\in\Gamma\}}$ such that
\begin{equation}\label{key identity for  cor 14}
\mu(U^{\gamma_1}(A)\cap \ldots \cap U^{\gamma_n}(A))=m(1_{\gamma_1^{-1} \Lambda}\cdot \ldots \cdot 1_{\gamma_n^{-1} \Lambda}),
\end{equation}
where $A=\{\omega\in X: \omega(e,e)=1\}$.
We now pass from the $\CHProb_\Gamma$-Roth dynamical system $(X,\mathcal{X},\mu,S,T)$, for which  \eqref{key identity for  cor 14} holds, to the corresponding $\OpProbAlgG$-dynamical system $(X_\mu,\bar\mu,\bar S,\bar T)$ by the deletion and abstraction process described in Section \ref{sec-canmodel}. This allows us to apply  Corollary \ref{maincor} and thereby obtain that
\begin{equation*}
\Theta:=\{\gamma\in \Gamma : \bar\mu([A]\wedge \bar T^\gamma([A])\wedge \bar S^\gamma \bar T^{\gamma}([A]))>0\}
\end{equation*}
is syndetic in $\Gamma$.
Thus for every $\gamma\in \Theta$ we can choose
$ \xi\in A\cap T^\gamma(A)\cap S^\gamma T^{\gamma}(A)$.
Since $A$ is open and $\xi\in \overline{\{U^{(\theta,\zeta)}(1_\Lambda):(\theta,\zeta)\in \Gamma\times\Gamma\}}$,
there exists $(\theta,\zeta)\in \Gamma\times\Gamma$ such that
\begin{equation*}
S^{\theta}T^{\zeta}(1_{\Lambda})=U^{(\theta,\zeta)}(1_{\Lambda})\in A\cap T^\gamma(A)\cap S^\gamma T^{\gamma}(A).
\end{equation*}
Therefore, $
(\theta,\zeta)$,  $(\gamma \theta,\zeta)$, and $(\gamma \theta,\gamma \zeta)$ are all in $ \Lambda$.
This finishes the proof.
\end{proof}

\section{Uniform syndeticity in the amenable ergodic Roth theorem}\label{sec-uniformity}

In this section, we prove our main application Theorem \ref{thm-syndec}.
In the proof we apply tools from ultralimit analysis (aka non-standard analysis).
In the following two lemmas, we establish some relations between lower Banach densities of a sequence of sets and the lower Banach density of their ultraproduct set.
These lemmas will be useful in the proof of Theorem \ref{thm-syndec} below.

\begin{lemma}\label{lem-dens2}
Let $\mathcal{G}$ be a uniformly amenable set of groups. Let $(\Gamma_n)$ be a sequence in $\mathcal{G}$.  Let $p$ be a non-principal ultrafilter on $\N$ and $\Gamma_*=\prod_{n\to p} \Gamma_n$ the ultraproduct of $(\Gamma_n)$.  Let $(m_n)$ be a sequence of invariant finitely additive probability measures $m_n:\mathcal{P}(\Gamma_n)\to [0,1]$.
Then we can associate to $(m_n)$ an invariant finitely additive probability measure $m:\mathcal{P}(\Gamma_*)\to [0,1]$.
\end{lemma}

We denote by $\st$ the standard part of a non-standard real number.

\begin{proof}
By \L os's theorem, the Loeb measure
\begin{equation*}
m(A_*)\coloneqq \st(\lim_{n\to p} m_n(A_n))
\end{equation*}
is an invariant finitely additive probability measure on the algebra of internal subsets $A_*=\prod_{n\to p} A_n$ of $\Gamma_*$ where $A_n\subset \Gamma_n$ for each $n$.
Define $M(1_{A_*})\coloneqq m(A_*)$, and extend $M$ to the closed linear hull $\mathcal{D}$ of $\{1_{A_*}: A_*\subset \Gamma_* \text{ internal}\}$ in $\ell^\infty(\Gamma_*)$ by linearity and continuity.
The closed subspace $\mathcal{D}$ majorizes $\ell^\infty(\Gamma_*)$ in the sense that for every $f\in \ell^\infty(\Gamma_*)$ there exists $g\in \mathcal{D}$ such that $f\leq g$ (we can take $g=\|f\|_\infty$ since $1\in \mathcal{D}$).
By  Silverman's Hahn-Banach extension theorem for invariant means \cite{silverman2,silverman1}, we can extend\footnote{This extension is  not unique in general, however this will not cause an issue later.} $M$ to an invariant mean on the whole space $\ell^\infty(\Gamma_*)$.
\end{proof}

\begin{lemma}\label{lem-density1}
Let $\mathcal{G}$ be a uniformly amenable set of groups. Let $(\Gamma_n)$ be a sequence in $\mathcal{G}$.  Let $p$ be a non-principal ultrafilter on $\N$ and $\Gamma_*=\prod_{n\to p} \Gamma_n$ the ultraproduct of $(\Gamma_n) $.
Let $A=\prod_{n\to p} A_n$ be an internal subset of $\Gamma_*$, where $A_n\subset \Gamma_n$ for each $n$.
Then we have
\begin{equation}\label{eq-dens1}
\st(\lim_{n\to p}\BD_{\Gamma_n}(A_n))=\inf\{\st(\lim_{n\to p} m_n(A_n))\colon m_n\in \mathcal{M}_n, \, n\in\N\},
\end{equation}
where $\mathcal{M}_n$ is the collection of invariant finitely additive probability measures on $\Gamma_n$ for each $n$.
\end{lemma}

\begin{proof}
First, we show ``$\leq$'' in \eqref{eq-dens1}.
By definition, $\BD_{\Gamma_n}(A_n)\leq m_n(A_n)$ for all $m_n\in\mathcal{M}_n$ for all $n$. Hence, $\lim_{n\to p} \BD_{\Gamma_n}(A_n)\leq \lim_{n\to p} m_n(A_n)$ for any sequence $(m_n)$ with $m_n\in\mathcal{M}_n$, and thus $\st(\lim_{n\to p} \BD_{\Gamma_n}(A_n))\leq \st(\lim_{n\to p} m_n(A_n))$.  We obtain ``$\leq$'' in \eqref{eq-dens1} upon taking the infimum over all possible sequences $(m_n)$ with $m_n\in \mathcal{M}_n$ for each $n$.

Second, we show ``$\geq$'' in \eqref{eq-dens1} by contradiction, that is, assume that we had $C<D$, where we denote by $C,D$ the left- and right-hand side of \eqref{eq-dens1}, respectively.
By definition, $\{n\in \N\colon \BD_{\Gamma_n}(A_n)< D\}\in p$ which implies that
\begin{equation*}
\{n\in \N\colon \exists \; m_n\in\mathcal{M}_n \text{ such that } m_n(A_n)< D\}\in p.
\end{equation*}
In particular, $\{n\in \N\colon m_n(A_n)< \st(\lim_{n\to p} m_n(A_n))\}\in p$, and it follows that $ \st(\lim_{n\to p} m_n(A_n))< \st(\lim_{n\to p} m_n(A_n))$ which is absurd. This proves the ``$\geq$'' part in~\eqref{eq-dens1}.
\end{proof}

We are now in a good position to prove Theorem \ref{thm-syndec}.

\begin{proof}
Let $\mathcal{G}$ be a uniformly amenable set of groups.
Suppose by contradiction that there exists $\ep>0$ such that for every $n$ there are $\Gamma_n\in\mathcal{G}$ , a  $\OpProbAlg_{\Gamma_n}$-Roth dynamical system $(X_n,\mu_n,T_{X_n},S_{X_n})$ with canonical concrete $\CHProb_{\Gamma_n}$-representation \begin{equation*}(\Stone(X_n),\Baire(\Stone(X_n)),\mu_{\Stone(X_n)},T_{\Stone(X_n)},S_{\Stone(X_n)}), \end{equation*}
and $E_n\in \Baire(\Stone(X_n))$ with $\mu_{\Stone(X_n)}(E_n)\geq \ep$ such that
\begin{equation}\label{eq-contra1}
\BD_{\Gamma_n}(\{\gamma\in \Gamma_n\colon \mu_{\Stone(X_n)}(E_n\cap T_{\Stone(X_n)}^\gamma E_n\cap S_{\Stone(X_n)}^\gamma T_{\Stone(X_n)}^\gamma E_n)>1/n\})\leq 1/n.
\end{equation}
Fix a non-principal ultrafilter $p$ on $\N$.
Construct the ultraproduct $\OpProbAlg_{\Gamma_*}$-Roth dynamical system $(X_\mu,\mu_{X_\mu},T,S)$ from the sequence
\begin{equation*}
(\Stone(X_n),\Baire(\Stone(X_n)),\mu_{\Stone(X_n)},T_{\Stone(X_n)},S_{\Stone(X_n)})
\end{equation*}
 by the recipe in Appendix \ref{sec-ultra}.
By construction, we have $\mu_{X_\mu}([E_*])\geq \ep$ where $E_*=\prod_{n\to p} E_n$.
By Corollary \ref{maincor}, there exists $\delta>0$ such that  \begin{equation}\label{eq-density}
 D\coloneqq \BD_{\Gamma_*}(
\{\gamma_* \in \Gamma_* \colon \mu_{X_\mu}([E_*] \wedge T_*^{\gamma_*} [E_*]\wedge S_*^{\gamma_*} T_*^{\gamma_*} [E_*])>\delta\})>0.
\end{equation}
Unwrapping all definitions, we have
\begin{align}\label{eq-synd}
&\{\gamma_* \in \Gamma_* \colon \mu_{X_\mu}([E_*] \wedge T_*^{\gamma_*} [E_*]\wedge S_*^{\gamma_*} T_*^{\gamma_*} [E_*])\} \\[2mm]
\notag &= \prod_{n\to p} \{\gamma_n \in \Gamma_n : \mu_{\Stone(X_n)}(E_n \cap T_{\Stone(X_n)}^{\gamma_n} E_n \cap S_{\Stone(X_n)}^{\gamma_n}T_{\Stone(X_n)}^{\gamma_n} E_n)\}.
\end{align}
Denote   $A_n=\{\gamma_n \in \Gamma_n : \mu_{\Stone(X_n)}(E_n \cap T_{\Stone(X_n)}^{\gamma_n} E_n \cap S_{\Stone(X_n)}^{\gamma_n}T_{\Stone(X_n)}^{\gamma_n} E_n)\}$.
It follows from Lemma \ref{lem-dens2} and \eqref{eq-synd} that
\begin{equation*}
D\leq \inf\{\st(\lim_{n\to p} m_n(A_n))\colon m_n: \mathcal{P}(\Gamma_n)\to [0,1],\, n\in \N\},
\end{equation*}
where the $m_n$ denote finitely additive invariant probability measures.
By Lemma~\ref{lem-density1},
\begin{align}\label{eq-contra}
& D\leq \\[2mm] \nonumber
 & \st(\lim_{n\to p} \BD_{\Gamma_n}(\{\gamma_n\in \Gamma_n\colon \mu_{\Stone(X_n)}(E_n\cap T_{\Stone(X_n)}^{\gamma_n} E_n\cap S_{\Stone(X_n)}^{\gamma_n} T_{\Stone(X_n)}^{\gamma_n} E_n)>\delta\})).
\end{align}
Define the set $R$ to be all $n\in \N$ that satisfy
\begin{equation*}
\BD_{\Gamma_n}(\{\gamma_n\in \Gamma_n\colon \mu_{\Stone(X_n)}(E_n\cap T_{\Stone(X_n)}^{\gamma_n} E_n\cap S_{\Stone(X_n)}^{\gamma_n} T_{\Stone(X_n)}^{\gamma_n} E_n)>\delta\})\leq \frac{1}{n}.
\end{equation*}
Then, by hypothesis \eqref{eq-contra1},  $R$ contains all but finitely many $n$.
Since the Fr\'echet filter is contained in any non-principal ultrafilter (see the beginning of Appendix \ref{sec-ultra}), we have $R\in p$.
Therefore it follows from \eqref{eq-contra} that $D$ must be zero, however this contradicts \eqref{eq-density}.
\end{proof}

\appendix

\section{Boolean algebras and the Stone representation theorem}\label{sec-BA}

A \emph{Boolean algebra} is a ring $(X, +, \cdot)$ with a multiplicative identity $1$ in which $x^2=x$ for every $x\in X$. We always assume the non-degeneracy condition $0\neq 1$.
A~prototypical example is $(\X,\Delta,\cap)$ where $X$ is any set and $\X \subset 2^X$ is an algebra of subsets of $X$, and $\Delta$ is the symmetric difference.
Its zero is the empty set $\emptyset$ and the multiplicative identity is $X$.
In particular, we have the trivial algebra $(\{\emptyset,X\},\Delta,\cap)$ which is ring-isomorphic to the finite field $(\mathbb{F}_2,+,\cdot)$. Given a boolean algebra $(X,+,\cdot)$ and $x,y\in X$, we set  $x\Delta y= x+y$, $x\wedge y =x\cdot y$,  $x\vee y = x+y+xy$ and denote $x\leq y$ if and only if $x\cdot y=x$.

A \emph{subalgebra} is a subring
of $X$ which contains its multiplicative identity.
A set $I\subset X$ is an ideal if and only if $0\in I$, $x \vee y\in I$ for all $x,y\in I$, and $x\in I$ whenever $x\leq y$ and $y\in I$. Note that while an ideal in a Boolean algebra is necessarily a subring, it constitutes a subalgebra only if it is $X$ itself.  Moreover, the quotient ring $X/I$ is a Boolean algebra called the \emph{quotient algebra}.

A map $f:X \rightarrow Y$ between two Boolean algebras $X$ and $Y$ is called a \emph{Boolean homomorphism} if it is a ring homomorphism,  that is, $f(x\Delta y) =f(x)\Delta f(y)$ and $f(x\wedge y) =f(x)\wedge f(y)$, and maps the multiplicative identity of $X$ to the multiplicative identity of $Y$. Note that $f(X)$ is a subalgebra of $Y$.

A Boolean algebra is called \emph{$\sigma$-complete} if every non-empty countable subset has a least upper bound.
An ideal $I$ of a Boolean algebra is called a \emph{$\sigma$-ideal} if every non-empty countable subset of $I$ has a least upper bound in $I$.
If $I$ is a $\sigma$-ideal $I$ in a $\sigma$-complete Boolean algebra $X$, then the quotient algebra $X/I$ is $\sigma$-complete as well.

Any abstract Boolean algebra can be represented by a concrete Boolean algebra of sets by Stone's representation theorem, as follows.
Consider the set $Z_X$ of all (non-zero) ring homomorphisms from $X$ to $\mathbb{F}_2$. The image $s(X)$ under the map
  $s:X \rightarrow 2^{X}$,  $x \mapsto s(x)=\{f\in Z_X: f(x)=1\}$ is a base of a topology on $Z_X$. The set $Z_X$ equipped with the topology generated by $s(X)$ is called the \emph{Stone space} of the Boolean algebra $X$. Moreover one can show that $s(X)$ corresponds to the set of all clopen (closed and open) subsets of $Z_X$. Therefore $Z_X$ is a totally disconnected space.  Moreover one can show that it is also compact Hausdorff.
By \emph{Stone's representation theorem}, the set of all clopen subsets of $Z_X$ equipped with the usual set operations is a Boolean algebra isomorphic to $X$, see \cite[Section 7]{koppelberg89} for a comprehensive introduction into this topological version of Stone duality.
The Stone space $Z_X$ can also be regarded as a closed subspace of a generalized Cantor space. More precisely, $Z_X$ is a closed subspace of $\{0,1\}^X=\mathbb{F}_2^X$ viewed as topological product space with $\{0,1\}$ endowed with the discrete topology (see \cite[Section 7]{koppelberg89} for details).
Any Boolean homomorphism $f:X\to Y$ between Boolean algebras $X$ and $Y$ can be uniquely represented as a continuous function $\hat{f}:Z_Y\to Z_X$ given by $\hat{f}(\alpha)=\alpha\circ f$ where $\alpha\in Z_Y$ is a Boolean homomorphism from $Y$ to $\mathbb{F}_2$.
This correspondence is a contravariant functor between the category Boolean algebras and Stone spaces, which  establishes a well-known equivalence of categories, known as Stone duality.

\section{Amenability, syndeticity, and uniform amenability}\label{sec-amsyn}
\subsection{Amenability and syndeticity}\label{appendix_B1}
A discrete group $\Gamma$ is said to be \emph{amenable} if it satisfies one of the following equivalent conditions:
\begin{itemize}
    \item[(i)]  (\emph{F\o lner condition}) For every  finite set $\Psi\subset \Gamma$ and every $0<\ep<1$ there exists a finite set $\Phi=\Phi(\ep,\Psi)\subset \Gamma$ such that
    \begin{equation} \label{folner:first}
         \max_{\gamma\in \Psi} |\Phi \Delta \gamma \Phi|\leq \ep |\Phi|.
    \end{equation}
    \item[(ii)] There exists a \emph{F\o lner net} for $\Gamma$, that is,  a net $(\Phi_\alpha)_{\alpha\in A}$ of non-empty finite subsets of $\Gamma$, such that
    \begin{equation*}
    \lim_{\alpha \in A} \frac{|\Phi_\alpha \Delta \gamma \Phi_\alpha|}{|\Phi_\alpha|}\to 0
    \end{equation*}
    for all $\gamma\in \Gamma$.
    \item[(iii)] There exists an \emph{invariant mean} for $\Gamma$, that is a positive linear functional $m: \ell^\infty(\Gamma)\to \R$ with the properties that $m(1)=1$ and $m(\gamma f)=m(f)$, where $(\gamma f)(\gamma'):=f (\gamma^{-1}\gamma')$ is the left-regular representation.
    \item[(iv)]  There exists an \emph{invariant finitely additive probability measure} $\mu:\mathcal{P}(\Gamma)\to [0,1]$, that is a finitely additive probability measure $\mu:\mathcal{P}(\Gamma)\to [0,1]$ such that $\mu(\gamma E)=\mu(E)$ for all $\gamma\in\Gamma$ and $E\subset \Gamma$, where $\gamma E=\{\gamma \tilde{\gamma}:\tilde{\gamma}\in E\}$.
\end{itemize}
See \cite{Leptin, paterson} for a proof of these equivalences.

One is often  interested when a certain subset of a discrete amenable group is syndetic. The following proposition gives a sufficient condition for a subset of a discrete amenable group  to be syndetic.

\begin{proposition}\label{prop-syndetic}
Let $\Gamma$ be a discrete amenable group.   Let $E$ be a subset of $\Gamma$.  If for every left F\o lner net $(\Phi_\alpha)_{\alpha\in A}$ for $\Gamma$ there exists $\alpha\in A$ such that $\Phi_\alpha\cap E\neq \emptyset$, then $E$ is syndetic.
\end{proposition}

\begin{proof}
Towards a contradiction assume that $E$ is not syndetic. Let $(\Phi_\alpha)_{\alpha\in A}$ be an arbitrary F\o lner net. Since we assumed that $E$ is not syndetic and each $\Phi_\alpha$ is finite, we must have that there exists $h_\alpha \in \Gamma \backslash \Phi_\alpha^{-1}E$ for each $\alpha\in A$. Now $(\Phi_\alpha h_\alpha)_{\alpha\in A}$ is a F\o lner net (as can be easily seen from translation invariance of Haar counting measure on $\Gamma$) such that $\Phi_\alpha h_\alpha \cap E=\emptyset$ contradicting the hypothesis, thus $E$ is syndetic.
\end{proof}

Bergelson, Hindman and McCutcheon established the following   characterization of syndeticity in discrete amenable groups in \cite[Theorem 2.7(a)]{bergelson-size}.

\begin{lemma}\label{lem-density-syndeticity}
Let $\Gamma$ be a discrete amenable group.
A subset $E\subset \Gamma$ is syndetic if and only if $\BD_\Gamma(E)>0$.
\end{lemma}
The proof of Lemma \ref{lem-density-syndeticity} in \cite{bergelson-size} relies on the notion of the lower Banach density defined in \eqref{def:lbd-gen}. For the integers, this definition is equivalent to  the definition  \eqref{def:lowerbd-Z}.

\begin{lemma}
For every $A\subset \Z$ we have
    \begin{align}\label{bd1}
    & \liminf_{b-a \to \infty}\frac{|A\cap  \{a,a+1,\ldots,b\}|}{b-a+1}\\[2mm]\label{bd2}
   &=\inf \Big \{\liminf_{n\to \infty} \frac{|A\cap F_n| }{|F_n|}: (F_n)_{n\in \Z} \text{ is a F\o lner sequence for $\Z$} \Big \}\\[2mm] \label{bd3}
    &=
   \inf\{ \nu(A): \nu \text{ is an invariant finitely additive probability measure on $\Z$} \}.
\end{align}
\end{lemma}

\begin{proof}
We start by proving that \eqref{bd1} equals  \eqref{bd2}. We follow the argument in \cite{Bergelson} which discusses an analogous statement for the upper Banach density.
First note that an arbitrary sequence of intervals
\begin{equation}
(\{a_n,\ldots, b_n\})_{n} \quad  \text{with}\quad b_n-a_n\to \infty \quad \text{ as  $\quad n\to \infty$}
\end{equation}
 is a F\o lner sequence for $\Z$.
Thus, \eqref{bd2} is no greater than \eqref{bd1}.

To see the reverse inequality, it suffices to prove that given any set $A\subseteq \Z$ and  any F\o lner sequence $(F_n)_n$  there is a sequence  $(x_n)_{n}$ such that \eqref{bd2} equals
\begin{equation}\label{bdproof}
\lim_{n\to \infty}\frac{|A\cap (x_n+F_n)|}{|F_n|}.\end{equation}
To show this, it suffices to show that given any $\beta>\eqref{bd2}$ and any finite set $F$, there exists $x\in \Z$ such that
\begin{equation*}
    \frac{| A\cap (x+F)|}{|F|}\leq \beta.
    \end{equation*}

Observe that if there is a F\o lner sequence $(G_n)_n$ such that for some $G\in \{G_n\}_n$,  one has
\begin{equation}\label{eq:G}
\frac{|A\cap (y+G)|}{|G|}\leq \beta\quad \text{for each}\quad y\in F,
\end{equation}
then we have
\begin{equation*} \sum_{x\in G} |A\cap (x+F)| = |\{(y,x)\in F\times G: y+x \in A \}| = \sum_{y\in F}|A\cap (y+G)| \leq \beta |G| |F|.
\end{equation*}
Dividing by $|G| |F|$ and argue by pigeonholing, we complete the proof.

We prove the existence of a F\o lner sequence $(G_n)_n$ such that for some $G\in \{G_n\}_n$ we have \eqref{eq:G}, by contradiction. Suppose then that for every F\o lner sequence  $(G_n)_n$ and every $n$ there exists $y_n\in F$ such that $|A\cap (y_n+G_{n})|/|G_{n}|> \beta$. Since $(G_n)_n$ is F\o lner, for any $\ep>0$, there is $n_0$ such that if $n\geq n_0$,  $|(y_n+G_n) \Delta G_n|/|G_n|\leq ilon$. This, in turn,  implies $\eqref{bd2} \geq \beta$, a contradiction.

The equivalence of \eqref{bd2} and \eqref{bd3} is classical and can  be derived   from \cite[Theorem 4.17]{paterson}\footnote{We thank Joel Moreira for providing us with this reference.}.
\end{proof}

\subsection{Uniform amenability}
The notion of uniform amenability  was  introduced by Keller \cite{keller} with the purpose of defining a notion of amenability that is stable under switching to non-standard models (ultrapowers) of a given group.
The condition stated in  \eqref{def:unifam}  is nowadays referred to as the {\em uniform F\o lner condition}.

The integers $\Z$ are an example of a uniformly amenable group. More generally, all solvable groups are uniformly amenable   \cite{keller,Dronka}. Note that the class of solvable groups includes examples of uncountable groups.
 Moreover, finite products of uniformly amenable groups, as well as subgroups and homomorphic images of a uniformly amenable group are all uniformly amenable. Extensions of a uniformly amenable group by a uniformly amenable group are  uniformly amenable. See \cite[Section 4]{keller} for these properties.

The uniform F\o lner condition is indeed a uniform version of the F\o lner condition given in \eqref{folner:first}. Hence, every uniformly amenable group is amenable.  The converse implication does not hold. The group $S_\infty$ of permutations of $\N$ which move only a finite number of elements is an example of a group which is amenable but not uniformly amenable, see \cite{Dronka}. See \cite{Wysoczanski} for another example.

Proposition \ref{prop-uniform-amenability} below characterizes  uniformly amenable groups and uniformly amenable sets of   groups via ultraproducts.

\section{Ultraproducts of measure-preserving dynamical systems}\label{sec-ultra}

An \emph{ultrafilter} on $\mathbb{N}$ is a non-empty collection $p$ of subsets of $\mathbb{N}$ satisfying the following properties:
\begin{itemize}
    \item[(i)] $\emptyset\not\in p$,
    \item[(ii)] $A\cap B\in p$ whenever $A,B\in p$,
    \item[(iii)] $B\in p$ whenever $A\in p, A\subset B$,
    \item[(iv)] for all $A\subset \mathbb{N}$ either $A\in p$ or $A^c\in p$.
\end{itemize}
Property (iv) distinguishes an ultrafilter from a filter.
An ultrafilter $p$ is said to be \emph{principal} if there is a non-empty set $A\subset \mathbb{N}$ such that $p=\{B\in\mathcal{P}(\N)\colon A\subset B\}$. An ultrafilter is \emph{non-principal} if it is not principal.  The existence of non-principal ultrafilters is only guaranteed by the axiom of choice.  More precisely, consider the Fr\'echet filter $\mathcal{F}$ which is the smallest filter containing all cofinal sets $\{n,n+1,\ldots\}$, $n\in \N$. By the Boolean prime ideal theorem, there exists an ultrafilter $p$ containing $\mathcal{F}$.
By construction, $p$ is non-principal.
On the other hand, any non-principal ultrafilter $p$ contains the Fr\'echet filter.  Indeed, since $p$ is an ultrafilter, it must contain either $\{n,n+1,\ldots\}$ or its complement by property (iv) above. But if it contained the complement of $\{n,n+1,\ldots\}$, then it would be a principal ultrafilter.

In what follows, the ultrafilter $p$ on $\N$ is fixed.
Let $\mathcal{G}$ be a uniformly amenable set of discrete groups.
For each $n\in\N$, let $\Gamma_n\in \mathcal{G}$, let $(X_n,\mu_n,T_{n})$ be a $\OpProbAlg_{\Gamma_n}$-dynamical system, and let
$(\Stone(X_n), \Baire(\Stone(X_n)), \mu_{\Stone(X_n)},T_{\Stone(X_n)})$ be the corresponding canonical model in $\CHProb_{\Gamma_n}$  (see Definition \ref{def-spaces}).

The aim of this appendix is to sketch the construction of a  $\OpProbAlg_{\Gamma_*}$-dynamical system induced by a  $\ConcProb_{\Gamma_*}$-dynamical system associated to an ultraproduct measure-preserving dynamical system of the sequence of systems
\begin{equation*}
(\Stone(X_n),\Baire(\Stone(X_n)), \mu_{\Stone(X_n)},T_{\Stone(X_n)}),
\end{equation*}
where $\Gamma_*$ is the ultraproduct group of the sequence $(\Gamma_n)$ of uniformly amenable groups.
The group $\Gamma_*$ is defined as the quotient group of $\prod_{n\in\N} \Gamma_n$ with respect to the equivalence relation $(\gamma_n)\sim (\tilde{\gamma}_n)$ whenever $\{n\in\N\colon \gamma_n=\tilde{\gamma}_n\}\in p$.

The following characterization of uniform amenability was established by Keller in \cite[Theorem 4.3 and Lemma 5.3]{keller}.

\begin{proposition}\label{prop-uniform-amenability}
A discrete group $\Gamma$ is uniformly amenable if and only if for every non-principal ultrafilter $p$ on $\mathbb{N}$ the ultrapower group
\begin{equation*}
\Gamma_*=\prod_{n\to p} \Gamma
\end{equation*}
is amenable (given the discrete topology).
Similarly, a set $\mathcal{G}$ of discrete groups is  uniformly amenable if and only if for every non-principal ultrafilter $p$ on $\mathbb{N}$ and for any sequence $(\Gamma_n)$ of groups in $\mathcal{G}$ the ultraproduct
\begin{equation*}
\Gamma_*=\prod_{n\to p} \Gamma_n
\end{equation*}
is amenable (given the discrete topology).
\end{proposition}
Similarly to $\Gamma_*$, we define the ultraproduct $X_*=\prod_{n\to p} \Stone(X_n)$ as the set of equivalence classes of elements of the product  $\prod_{n\in \N} \Stone(X_n)$ with respect to the equivalence relation  $(x_n)\sim (y_n)$ defined by $\{n\in\N \colon x_n=y_n\}\in p$.
Let
\begin{equation*}
\mathcal{A}=\{\prod_{n\to p} E_n: (E_n)\in \prod_{n\in\N} \Baire(\Stone(X_n))\}.
\end{equation*}

Using the ultrafilter axioms, one can verify that $\mathcal{A}$ is an algebra of subsets of $X_*$.

The Loeb  premeasure  $\mu_*:\mathcal{A}\to [0,1]$ is defined by
\begin{equation*}
\mu_*(\prod_{n\to p} E_n):=\st(\lim_{n\to p} \mu_{\Stone(X_n)}(E_n)),
\end{equation*}
where $\st$ denotes the standard part of a non-standard real. Using the countable saturation property and Carath\'eodory's extension theorem, $\mu_*$ can be extended to a measure $\mu$ on the $\sigma$-algebra $\mathcal{X}=\sigma(\mathcal{A})$ generated by $\mathcal{A}$ in $X_*$. Let $(X_\mu,\bar\mu)$ denote the probability algebra of $(X,\mathcal{X},\mu)$.

Let us denote by $\Aut(\mathcal{A},\mu_*)$ the automorphism group of $(\mathcal{A},\mu_*)$, that is the group of Boolean isomorphisms $f:\mathcal{A}\to \mathcal{A}$ such that $\mu_*(E)=\mu_*(f(E))$ for all $E\in\mathcal{A}$.
By chasing definitions, one can check that the sequence $(T_{\Stone(X_n)})$ of measure-preserving continuous actions induces a concrete action $T_*:\Gamma_*\to \Aut(\mathcal{A},\mu_*)$ by defining
\begin{equation*}
(T_*)^{\gamma_*}(\prod_{n\to p} E_n)\coloneqq \prod_{n\to p} T^{\gamma_n}_{\Stone(X_n)}(E_n)
\end{equation*}
for all $\gamma_*=[(\gamma_n)]\in \Gamma_*$ and $\prod_{n\to p} E_n\in\mathcal{A}$.
By construction, $\mathcal{A}$ is an algebra of sets which is dense in $\mathcal{X}$ with respect to the pseudo-metric $d(E,F)=\mu(E\Delta F)$ on $\mathcal{X}$.
We define the abstract action $\bar T:\Gamma_*\to \Aut(X_\mu,\bar\mu)$ by
\begin{equation*}
\bar T^{\gamma_*}([E])\coloneqq \bigvee_n [T_*^{\gamma_*}(E_n)]
\end{equation*}
where $[E]$ denotes the equivalence class of $E\in \X$ in $X_\mu$, and $(E_n)$ is a sequence in $\mathcal{A}$ such that $\mu(E_n\Delta E)\to 0$ as $n$ tends to infinity. Observe that  the definition of  $\bar T^{\gamma_*}([E])$ is independent of the choice of representatives and the approximating sequence.
We thus obtain a $\OpProbAlg_{\Gamma_*}$-dynamical system $(X_\mu,\bar\mu,\bar T)$.

Finally, suppose that for each $n\in\N$, $(X_n,\mu_n,S_n)$ is another  $\OpProbAlg_{\Gamma_n}$-dynamical system  such that $S_n$ and $T_n$ commute. Construct $(X_\mu,\bar\mu,\bar S)$ analogously to $(X_\mu,\bar\mu,\bar T)$ as before. Then $S$ and $T$ commute (which is easily seen by first verifying commutativity of $S_*$ and $T_*$ on $\mathcal{A}$), and therefore $(X_\mu,\bar \mu,\bar S,\bar T)$ becomes a $\OpProbAlg_{\Gamma_*}$-Roth dynamical system.

\begin{remark}
In \cite[\S 3,4]{ultraproducts-mps}, Conlon, Kechris and Tucker-Drob give an ultraproduct construction of a sequence $(X_n,\mathcal{X}_n,\mu_n,T_n)$ of $\ConcProb_\Gamma$-dynamical systems where $(X_n,\mathcal{X}_n,\mu_n)$ are standard Borel probability spaces and $\Gamma$ is a fixed countably infinite group, which is based off a construction of Elek and Szegedy \cite{elek-szegedy} for finite probability spaces.
They define a pointwise action of $\Gamma$ (and not its ultrapower) on the Loeb probability space associated to the sequence $(X_n,\mathcal{X}_n,\mu_n)$ by  taking the ultralimit of the sequence $(T_n)$ of the pointwise actions.
\end{remark}

\section{A conditional Heine-Borel covering lemma}\label{sec-canalysis}

This appendix is devoted to the following technical lemma which is needed in the proof of Theorem~\ref{thm-uncountableroth}.
Throughout this section,  $\Gamma$ is some group and $\pi:(X,\mu,T)\to (Y,\nu,S)$ a $\OpProbAlgG$-factor map.
We need the following notation.
\begin{align*}
    \langle f,g\rangle_{X|Y}&:=\E(f \bar{g}|Y), \quad f,g\in L^2(X),\\[2mm]
    \|f\|_{X|Y}&:=\E(|f|^2|Y)^{1/2}, \quad f\in L^2(X),
\end{align*}
where $\bar{g}$ indicates complex conjugation.
\begin{lemma}\label{lem-heineborel}
Suppose $\mathcal{M}\subset L^2(X)$ is a finitely generated, closed, and $\Gamma$-invariant $L^\infty(Y)$ submodule of $L^2(X)$. Let $f\in \mathcal{M}$ be such that $\|\|f\|_{X|Y}\|_{L^\infty(Y)}<\infty$.
Then for every $\ep>0$ there exists a finite set $\mathcal{N}\subset \mathcal{M}$ such that for every $\gamma\in\Gamma$,
\begin{equation}\label{eq:hb}\min_{h\in\mathcal{N}}\|T^\gamma(f)-h\|_{X|Y}\leq \ep.\end{equation}
\end{lemma}
First we establish the following auxiliary result.
\begin{proposition}[A conditional Gram-Schmidt process]\label{prop-onb}
Let $\mathcal{M}$ be a finitely generated and closed $L^\infty(Y)$ submodule of $L^2(X)$.
Then there exist a partition\break $(E_j)_{j=0,\ldots, m}$ of $\Stone(Y)$ and a family $(\mathcal{M}_j)_{j=1,\ldots,m}$ of finite subsets of $\mathcal{M}$ satisfying the following properties.
\begin{itemize}
    \item[(i)]
{\small\begin{equation*}
\mathcal{M}=\left\{ \sum_{j=1}^{m} \left(\sum_{u_j\in \mathcal{M}_j} a_{u_j} u_j\right) 1_{E_j}\colon a_{u_j}\in L^\infty(Y) \text{ for all } u_j\in \mathcal{M}_j \text{ and } j=1,\ldots,m\right\}.
\end{equation*}}\item[(ii)] $\|u\|_{X|Y}=1$ on $E_j$ for all $u\in\mathcal{M}_j$ and $j=1,\ldots,m$.
    \item[(iii)] $\langle u, u'\rangle_{X|Y}=0$ on $E_j$ for all distinct $u,u'\in\mathcal{M}_j$ and $j=1,\ldots,m$.
\end{itemize}
Notice that (i), in particular, implies that $\mathcal{M}=\{0\}$ on $E_0$.
\end{proposition}

\begin{proof}
Since $\mathcal{M}$ is a finitely generated $L^\infty(Y)$ submodule of $L^2(X)$, there are $f_1,\ldots,f_n\in L^2(X)$ such that
\begin{equation*}
\mathcal{M}=\left\{\sum_{i=1}^n a_i f_i: a_i\in L^\infty(Y), \, i=1,\ldots,n\right\}.
\end{equation*}
Let
\begin{equation*}
u_1=
\begin{cases}
\frac{f_1}{\|f_1\|_{X|Y}} & \text{on } \{\|f_1\|_{X|Y}>0\}, \\[2mm]
0 & \text{else.}
\end{cases}
\end{equation*}
We need to justify why $u_1\in\mathcal{M}$ as $1/\|f_1\|_{X|Y}1_{\{\|f_1\|_{X|Y}>0\}}$ may not be in $L^\infty(Y)$.
For $N\in\N$, let
\begin{equation*}
u_1^N=
\begin{cases}
\frac{f_1}{\|f_1\|_{X|Y}} & \text{on } \{1/N\leq \|f_1\|_{X|Y}\leq N\}, \\[2mm]
0 & \text{else.}\end{cases}
\end{equation*}
Then $u_1^N$ converges to $u_1$ almost surely as $N$ tends to infinity and thus also in $L^2(X)$ by dominated convergence. Since $\mathcal{M}$ is $L^2$ closed we have showed $u_1\in \mathcal{M}$.

We define the $u_i$ for $i>1$ inductively as follows.  Suppose we already have $u_1,\ldots,u_k$ with $k\leq n-1$.  Then  set $g_{k+1}=f_{k+1}-\sum_{i=1}^k \langle f_{k+1},u_i\rangle_{X|Y} u_i$ and define
\begin{equation*}
u_{k+1}=
\begin{cases}
\frac{g_{k+1}}{\|g_{k+1}\|_{X|Y}} & \text{on } \{\|g_{k+1}\|_{X|Y}>0\}, \\[2mm]
0 & \text{else.}
\end{cases}
\end{equation*}
By a similar approximation as in the case of $u_1$, one can show that $u_{k+1}$ is an element of $\mathcal{M}$ (where we now have to approximate first $\langle f_i,u_i\rangle_{X|Y}$ in $L^\infty(Y)$, then $g_{k+1}$ and finally $u_{k+1}$).

Denote by $F_i=\{\|g_{i}\|_{X|Y}>0\}$ and $F_{i+n}=\{\|g_{i}\|_{X|Y}>0\}^c$ for all $i=1,\ldots, n$.
Form all finite intersections $F_{i_1}\cap F_{i_2}\cap \ldots \cap F_{i_k}$ with $1\leq i_1<i_2<\ldots<i_k\leq 2n$ for some $1\leq k\leq 2n$. Let $\mathcal{E}_0$ denote the collection of such finite intersections whose measure is positive. Then $\mathcal{E}_0$ forms a partition of $\Stone(Y)$. Pick an element $E=F_{i_1}\cap F_{i_2}\cap \ldots \cap F_{i_k}\in \mathcal{E}_0$ and let $\mathcal{M}_{E}=\{u_{i_t}: i_t\leq n\}$.
Let $\mathcal{E}_1$ denote the collection of elements $E$ of $\mathcal{E}_0$ such that $\mathcal{M}_E\neq \emptyset$.
Now enumerate the elements of $\mathcal{E}_1$ by $E_1,\ldots, E_m$ and correspondingly write $\mathcal{M}_j=\mathcal{M}_{E_j}$ for $j=1,\ldots,m$. Set $E_0=(\bigcup_{j=1}^m E_j)^c$.
By construction, $(E_j)_{j=0,\ldots,m}$ and $(\mathcal{M}_j)_{j=1,\ldots,m}$ satisfy the desired properties (i), (ii), and (iii).
\end{proof}

We can prove our conditional Heine-Borel covering lemma.

\begin{proof}[Proof of Lemma \ref{lem-heineborel}]
Suppose that $\mathcal{M}\subset L^2(X)$ is a finitely generated, closed and $\Gamma$-invariant $L^\infty$ submoldule of $L^2(X)$. Let $(E_j)_{j=1,\dots,m}$ and $(M_j)_{j=1,\dots,m}$ be as in Proposition \ref{prop-onb}.
By assumption, for all $\gamma\in \Gamma$
\begin{equation}\label{eq-bound8}
\|T^\gamma(f)\|_{X|Y}=S^\gamma(\|f\|_{X|Y})\leq C
\end{equation}
for some constant $C>0$.
For each $\gamma\in\Gamma$, we have
\begin{equation*}
\|T^\gamma(f)\|_{X|Y}=\sum_{j=1}^m \left(\sum_{u_j\in \mathcal{M}_j} |a_{u_j,j}^\gamma|^2\right)^{1/2}1_{E_j}
\end{equation*}
for some $a_{u_j,j}^\gamma\in L^\infty(Y)$. By \eqref{eq-bound8} we have $\left(\sum_{u_j\in \mathcal{M}_j} |a_{u_j,j}^\gamma|^2\right)^{1/2}\leq C$ for all $j$ and $\gamma$.

Hence, for any fixed $\ep>0$ , one can find for every $1\leq j\leq m$ finitely many vectors $b_{1,j},\ldots, b_{k_j,j}\in L^\infty(Y)^{|\mathcal{M}_j|}$ with $b_{p,j}=(b_{1,p,j},\ldots,b_{|\mathcal{M}_j|,p,j})$ for $p=1,\ldots,k_j$ such that
\begin{equation*}\begin{split}
&\left\{a\in L^\infty(Y)^{|\mathcal{M}_j|}: \left(\sum_{q=1}^{|\mathcal{M}_j|} |a_q|^2\right)^{1/2} \leq C\right\}\\[2mm] & \subset \bigcup_{p=1}^{k_j}  \left\{a\in L^\infty(Y)^{|\mathcal{M}_j|}: \left(\sum_{q=1}^{|\mathcal{M}_j|} |a_q-b_{q,p,j}|^2\right)^{1/2} \leq \ep\right\}.
\end{split}
\end{equation*}

Let $\mathcal{N}$ be the collection of all functions \begin{equation*} \sum_{j=1}^m \left(\sum_{q=1}^{|\mathcal{M}_j|}b_{q,p{_j},j} u_j \right)1_{E_j}\end{equation*}  for some choice $1\leq p_j\leq k_j$ for each $j$. Note that $\mathcal{N}$ is finite, and
 by construction $\mathcal{N}\subset\mathcal{M}$ satisfies \eqref{eq:hb}.
\end{proof}

\begin{remark}
The results of this section are inspired by conditional analysis and conditional set theory in \cite{filipovic2009separation,cheridito2015conditional,drapeau2016algebra}.
The existence of a conditional orthonormal basis for certain $L^0$ submodules of $(L^0)^d,  d\geq 1$ via a conditional Gram-Schmidt process  is established in \cite[Section 2]{cheridito2015conditional}, where $L^0$ denotes the algebra of equivalence classes of all complex measurable functions.  A conditional version of the Heine-Borel theorem within conditional set theory is established in \cite[Theorem 4.6]{drapeau2016algebra}.
It is crucial in the conditional analysis of $L^0$ modules to assume a closedness property under countable gluings which is referred to as $\sigma$-stability \cite{cheridito2015conditional} or stability under countable concatenations \cite{filipovic2009separation,drapeau2016algebra}.

A main difference in our analysis to the previously cited articles is that we work with $L^\infty$ submodules of $L^2$ spaces rather than with the larger $L^0$ modules.
 However $L^\infty$ modules do not satisfy this countable gluing property in general. We still manage to develop some useful portion of conditional analysis for the smaller $L^\infty$ modules by additionally requiring that these modules are finitely generated and closed in the $L^2$ topology. These requirements are naturally satisfied in the context of compact extensions in structural ergodic theory.
\end{remark}

\end{document}